\documentclass[11pt]{article}
        \usepackage{amsmath}
        \usepackage{amsfonts}
        \usepackage{amsthm}
        \usepackage{comment}
        \usepackage{url}
        \usepackage{color}
        \usepackage{enumitem}
        \setlist[enumerate]{font=\normalfont}
        \usepackage{mathtools}
		\usepackage{geometry}

        \usepackage{latexsym}
        \usepackage{amssymb}
		\usepackage{graphicx}        
		\usepackage{float}
        
        \usepackage[colorinlistoftodos]{todonotes}
		\usepackage[colorlinks=true, allcolors=blue]{hyperref}
		\usepackage{subcaption}

        \usepackage{eucal}

		\usepackage{natbib}
		\setcitestyle{numbers,open={[},close={]}}
		
        \theoremstyle{plain}
        \newtheorem*{theorem*}{Theorem}
        \newtheorem{theorem}{Theorem}[section]
        
        \newtheorem{lemma}[theorem]{Lemma}

        \theoremstyle{definition}
        \newtheorem{definition}[theorem]{Definition}
        \newtheorem{example}[theorem]{Example}
        
        \newtheorem*{example*}{Example}

        \theoremstyle{remark}
        \newtheorem{remark}[theorem]{Remark}
        \newtheorem*{remark*}{Remark}

\newcommand{\PP}{\mathbb{P}}
\newcommand{\RR}{\mathbb{R}}

\definecolor{red}{rgb}{0.7,0.15,0.15}
\definecolor{green}{rgb}{0,0.5,0}
\definecolor{blue}{rgb}{0,0,0.7}
\hypersetup{colorlinks, linkcolor={red},citecolor={green}, urlcolor={blue}}

\newcommand{\footremember}[2]{%
   \footnote{#2}
    \newcounter{#1}
    \setcounter{#1}{\value{footnote}}%
}
\newcommand{\footrecall}[1]{%
    \footnotemark[\value{#1}]%
}

\author{Hern\'an {\sc Garc\'ia}\footremember{uniandes}{ Departamento de matem\'aticas, Universidad de los Andes. Carrera $1^{\rm ra}\#18A-12$, Bogot\'a, Colombia. jh.garcia1776@uniandes.edu.co.; mj.junca20@uniandes.edu.co; mvelasco@uniandes.edu.co }  \and Camilo {\sc Hern\'andez} \footnote{Industrial Engineering and Operations Research Department, Columbia University. $500$ West $120$-th street, New York, NY 10027. USA. camilo.hernandez@columbia.edu.}
\and Mauricio {\sc Junca }\footrecall{uniandes} \and Mauricio {\sc Velasco}\footrecall{uniandes} }
 
\date{\today}

\title{Approximate super-resolution of positive measures in all dimensions.}

\begin{document}

\maketitle

\begin{abstract}
We study the problem of reconstructing a positive discrete measure on a compact set $K\subseteq \RR^n$ from a finite set of moments (possibly known only approximately) via convex optimization. We give new uniqueness results, new quantitative estimates for approximate recovery and a new sum-of-squares based hierarchy for approximate super-resolution on compact semi-algebraic sets.

 \vspace{5mm}


\end{abstract}

\section{Introduction}

Let $K\subseteq \RR^n$ be a compact set and let $V$ be a finite-dimensional vector space of continuous real-valued functions on $K$. If $L: V\rightarrow \RR$ is linear and $\mu$ is a finite, positive borel measure on $K$ then $\mu$ {\it represents} $L$  in V if $L(f)=\int_K fd\mu$ for all $f\in V$. In this article we study the {\it discrete reconstruction problem} which, given a representable operator $L$, asks us to find a positive discrete measure $\mu^*:=\sum_{i=1}^k c_i\delta_{x_i}$ with $c_i\geq 0$ and $x_i\in K$ which represents $L$ on $V$. 

Under very general conditions, such measures $\mu^*$ exist (see Lemma~\ref{Lem: basicTM} for details). Moreover, constructing explicit solutions $\mu^*$ is useful in a wide variety of applications, for instance:
\begin{enumerate}
\item Polynomial optimization: via the method of moments proposed by \citeauthor*{L} \cite{L} one can define an operator $L$ such that every representing measure is supported on minimizers of a given multivariate polynomial.
\item Numerical integration: any discrete representing measure $\mu^*$ gives us a cubature rule~\cite{L2} for computing integrals of functions in $V$ with respect to the measure $\mu$ via evaluation.
\item Optimal control theory: optimal control problems can be reformulated as problems on occupation measures as in~\cite{LHPT}. Any discrete measure representing optima gives us explicit optimal control policies. 
\end{enumerate}

A celebrated approach to solve the reconstruction problem goes by the name of {\it superresolution} (see ~\citeauthor*{CandesFernandezGranda} \cite{CandesFernandezGranda} \cite{FernandezGranda}) or of {\it Beurling minimal interpolation} (see~\citeauthor{DeCastroGamboa}~
\cite{DeCastroGamboa} \cite{AzaisDeCastroGamboa}) and consists of finding a minimizer $\mu^*$ of the total variation norm in the set $\mathcal{S}(K)$ of all signed Borel measures on $K$. More precisely, letting  $\|\mu\|_{\rm TV}:=\sup \int_K g d\mu$ as $g$ runs over all continuous functions $g$ on $K$ with $\|g\|_{\infty}\leq 1$ we want to solve the problem

\begin{equation}
\label{prob: TV}
\min_{\nu\in \mathcal{S}(K)} \|\nu\|_{\rm TV}\text{ : $\forall f\in V\left(\int_Kfd\nu=L(f)\right)$ }
\end{equation}

There is a wealth of foundational results about superresolution in dimension one. Motivated by applications, the objective of this article is to extend some of these basic results to the higher-dimensional polynomial setting (i.e. when $n>1$). More precisely, throughout the article we assume that our measures are positive and real-valued and that the vector space of functions $V:=V_{\leq d}$ consists of the set of polynomials of degree at most $d$ in $\RR ^n$. 

In this setting the most basic question we can ask is that of uniqueness: Given a positive discrete measure $\mu$ defining an operator $L_{\mu}(f):=\int_K f d\mu$, when can we uniquely recover $\mu$ from $L$ using superresolution?   This question leads to the following new numerical invariant of finite sets

\begin{definition}
For a finite set $X\subseteq K\subseteq \RR ^n$ define the {\it uniqueness degree} $d(X)$ as the smallest integer $d$ such that for every positive discrete measure $\mu$ supported on $X$ problem~(\ref{prob: TV})  has a unique solution when $V:=V_{\leq d}$ and $L:=L_{\mu}$.
\end{definition}
By a Theorem of De Castro and Gamboa~\cite[Theorem 2.1]{DeCastroGamboa} we know that for every $X\subseteq \RR$ of cardinality $k$ the uniqueness degree is given by $d(X)=2k$. Our first result is a generalization of this Theorem to higher-dimension. Recall that a finite set of points $X\subseteq \RR^n$ has an ideal $I(X)$ consisting of all polynomials vanishing on $X$, a generator degree $g(X)$ defined as the maximum degree of a minimal generator of $I(X)$ and an interpolation degree $i(X)$ defined as the minimum degree $d$ such that every real-valued function on $X$ is given by the restriction to $X$ of a polynomial of degree at most $d$. We have,

\begin{theorem}\label{thm: ExactUniqueness} If $X\subseteq K\subseteq \RR^n$ is a finite set then the following inequalities hold: 
\begin{enumerate}
\item $d(X)\leq \max(2g(X),i(X))$. In particular, if $X$ has cardinality $k$ and is not contained in any hyperplane in $\RR^n$ then  $d(X)\leq 2(k-n+1)$.

\item If $X\subseteq K^{\circ}$ then $d(X)\geq \ell(X)$ where $\ell(X)$ is the smallest degree of a hypersurface which is singular at all points of $X$.
\end{enumerate}
\end{theorem}

Its is easy to see that both inequalities in the Theorem agree in the one-dimensional setting implying the result of De Castro and Gamboa. The previous Theorem highlights the enormous differences between superresolution in one and in more dimensions. Whereas in one-dimension the uniqueness degree depends only on the cardinality of the set of points, in higher-dimension this is not the case and this degree is determined by the commutative algebra of the ideal of the set of points. In Section~\ref{Sec: Uniqueness} we show that Theorem~\ref{thm: ExactUniqueness} is often sharp and that there exist very different behaviors of $d(X)$ for sets of points {\it of the same cardinality} even in dimension two (see Remark~\ref{rmk: 10pts}).

In applications one is typically interested in measures whose support $X$ is not an arbitrary set of points but rather a {\it generic} set of points $X=\{p_1,\dots, p_k\}$, meaning that $(p_1,\dots, p_k)$ lie in the complement of a proper algebraic subset of $\left(\RR^n\right)^k$ (see Section~\ref{sec: generic} for details). For such sets the uniqueness degree should only depend on the cardinality and we can specialize the upper bounds from the previous Theorem obtaining

\begin{theorem}\label{thm: GenericExactUniqueness} If $X$ is a generic set of $k$ points in $\RR^n$ then the following inequalities hold:
\begin{enumerate}
\item $d(X)\leq 2(e+1)$ where $e$ is the smallest integer for which the inequality $k\leq \binom{n+e}{e}$ holds.
\item $d(X)\geq \ell$ where $\ell$ is the smallest integer for which the inequality $k\leq \frac{1}{n+1}\binom{n+\ell}{n}$ holds.
\end{enumerate}
\end{theorem}

There are several approaches for solving the optimization Problem~(\ref{prob: TV}): this can be done either via discretization as in~\cite{Donoho} (although it is known that this approach works poorly for closed spaced points~\cite{FW}), via a semidefinite formulation of the dual problem as in~\cite{FernandezGranda,TBR} or via sum-of-squares hierarchies as De Castro, Gamboa, Henrion and Lasserre propose in~\cite{DeCastroGamboaHenrionLasserre}. 

Since we are working in the context of reconstructing positive measures (and not signed measures) one can also use a simple sum-of-squares relaxation which we prove is guaranteed to work for degrees above the upper bound of Theorem~\ref{thm: ExactUniqueness} (see Section~\ref{Sec: moments} for details). This result highlights a second fundamental difference between the one-dimensional and higher-dimensional setting. Whereas nonnegative univariate polynomials coincide with sums-of-squares this correspondence is no longer true in general in higher-dimensions. This phenomenon is well understood geometrically~\cite{BSVJams} but leads to additional algorithmic difficulties when $n\geq 2$. Nevertheless,  our numerical examples (see Section~\ref{Sec: numerics}) show that the simple moments relaxation works well in practice when we have exact knowledge of the moments of the unknown measure. 

In many applications of the measure reconstruction problem, however, the moments of the measure we wish to reconstruct are known only approximately. More precisely, we fix a basis $\phi_1,\dots, \phi_m$ for $V$ and would like to recover a point measure $\mu$ from a known vector $y$ with components given by $y_i:=\int_K\phi_i d\mu +\epsilon_i$ where $\epsilon:=(\epsilon_1,\dots, \epsilon_m)$ is a noise term bounded by a known value $\delta$ i.e. $\|\epsilon\|_2\leq \delta$ . A very significant contribution in this setting is the work of Azais, De Castro and Gamboa~\cite{AzaisDeCastroGamboa} who give quantitative estimates for the error when the recovery mechanism is to solve the following {\it Beurling Lasso} (BLASSO) optimization problem: 

\begin{equation}
\label{prob: BLASSO}
\min_{\nu\in \mathcal{S}(K)} \|\nu\|_{\rm TV} \text{ : } \left\|\left(\int_K\phi_i d\nu-y_i\right)_{i=0,\dots,m}\right\|_2\leq \delta
\end{equation}

Our next result gives quantitative localization bounds for problem~(\ref{prob: BLASSO}) in all dimensions. Its proof is a combination of the ideas of Azais, De Castro and Gamboa together with the explicit construction of $L^2$-optimal approximations to Dirac delta functions and some basic commutative algebra (see Section~\ref{Sec: QuantitativeApprox}). In order to describe the result we introduce the following notation: If $\Delta$ is a discrete measure and $z\in K$ we will write $\Delta(z)$ to mean the coefficient of $\delta_z$ in the unique decomposition of $\Delta$ as a sum of Dirac measures. We will write $d(X,z)$ for the euclidean distance between a point $z$ and a set $X$ and write $N(X,\delta)$ (resp. $F(X,\delta)$) for the set of points which are at distance at most (resp. at least) $\delta$ from $X$. We fix a basis $\phi_1,\dots, \phi_T$ of $V_{\leq d}$ which we assume to be orthonormal with respect to some probability measure on $K$.

\begin{theorem}\label{thm: ApproxMain} Let $\mu$ be any positive discrete measure supported on a finite set $X\subseteq K\subseteq \RR^n$ and let $\hat{\Delta}$ be a discrete minimizer of~(\ref{prob: BLASSO}) with $y_i:=\int\phi_id\mu +\epsilon_i$ and $\|\epsilon\|_2\leq \delta$. If $I(X)$ admits a set of $s$ generators of degree $\leq g$ and $d\geq 2(m-1)g$ then there exist a positive constant $D$ such that the following statements hold for all sufficiently large even integers $m$:
\begin{enumerate}
\item If $z\in K$ is such that $\hat{\Delta}(z)> 4s\delta$ then $d(X,z)\leq c_0$ where $c_0= \frac{1}{m}\sqrt\frac{3!}{D}$.
\item The following inequalities hold:
\[
\sum_{z\in N\left(X,c_0\right), \hat\Delta(z)>0} \hat{\Delta}(z) d(X,z)^2\leq \left(\frac{3!4s}{m^2D}\right)\delta\]
\[\sum_{z\in F\left(X,c_0\right), \hat\Delta(z)>0} \hat{\Delta}(z) \leq 2s\delta\]
\[\sum_{z: \hat\Delta(z)<0}|\hat{\Delta}(z)|\leq 2\delta\]
\item If $x^*\in X$ then the following inequality holds: 
\[\left| \mu(x^*)-\sum_{z: d(z,x^*)\leq c_0}\hat\Delta(z)\right|\leq \frac{1}{m}\|\mu\|_{\rm TV} + \left(2(s+1)+ \frac{3!2s\pi^2}{{\rm diam}(K)^2D}\right)\delta.\]
\end{enumerate}
\end{theorem}

In words, the previous Theorem says that the recovered measure $\hat\Delta$ has no large spikes far from those of $\mu$ (parts $(1)$, $(2)$) and furthermore that it has spikes near every support point of $\mu$ whose coefficients approximate those of $\mu$ rather well (part $(3)$) when $d$ is sufficiently large. In particular, it gives us explicit dependencies on the quality of our approximation as a function of the degree $d$ and the error size $\delta$.

The explicit determination of the constants appearing in the previous Theorem is, in general, a challenging problem which depends on the geometry of the support set $X$. In Example~\ref{Exo: Grids} we give an estimate for these quantities when the measures are supported on any grid in $\RR^n$. As is the case in one-dimensional super-resolution the key determinants of these constants end up being suitable measures of the distance between support points. 

Finally, in order to apply Theorem~\ref{thm: ApproxMain} we must be able to solve the (infinite-dimensional) optimization problem~(\ref{prob: BLASSO}). Our next Theorem recasts~(\ref{prob: BLASSO}) as a finite-dimensional convex optimization problem extending the main results of De Castro, Gamboa, Henrion and Lasserre in~\cite{DeCastroGamboaHenrionLasserre} to the approximate recovery problem.

\begin{theorem}\label{thm: dual} The optimal value of~(\ref{prob: BLASSO}) coincides with the optimal value of the following finite-dimensional convex optimization problem
\begin{equation}\label{prob: BLASSO2}
\sup_{(\vec{a},b)\in \RR^n\times \RR} \left\{\langle \vec{a}, y \rangle-b\delta : P=\sum_{i=1}^m a_i\phi_i,\|P\|_{\infty}\leq 1, \|\vec{a}\|_2\leq b\right\}
\end{equation}
\end{theorem}

Next we propose a hierarchy of semidefinite programs for solving~(\ref{prob: BLASSO2}) when $K$ is semialgebraic and explicitly bounded and $V\subseteq \RR[\vec{x}]:=\RR[x_1,\dots, x_n]$. To describe the hierarchy we will need the following basic definition. For $g_1,\dots, g_t\in \RR[\vec{x}]$ and $e\in \mathbb{Z}_+$ recall that the quadratic module of degree $e$ of $g_1,\dots, g_t$ is given by
\[Q_e(g_1,\dots,g_t)=\left\{f\in \RR[\vec{x}]: \exists (s_i)_{i=0,1,\dots, t} \text{  such that $f=s_0+\sum_{i=1}^tg_is_i$}\right\}.\]
where the $s_i\in \RR[\vec{x}]$ are sums-of-squares of polynomials of degree bounded by $e$. Henceforth we let $\vec{\phi}=(\phi_1,\dots, \phi_T)$ be the vector whose components are our chosen basis for $V$.

\begin{theorem}\label{thm: hierarchy} Suppose $K=\{x\in \RR^n: g_1(x)\geq 0,\dots, g_t(x)\geq 0\}$ for some $g_i\in \RR[x_1,\dots, x_n]$ and assume there exist  positive integers $N,e$ such that $N-\|x\|_2^2\in Q_{e}(g_1,\dots, g_t)$.  If $\alpha_s$ denotes the number 
\[\alpha_s:=\sup_{(\vec{a},b)\in \RR^n\times \RR}\left\{\langle\vec{a},y\rangle -b\delta: 1-\langle \vec{a},\vec{\phi}\rangle, 1+\langle \vec{a},\vec{\phi}\rangle \in Q_s(g), \|\vec{a}\|_2\leq b \right\}.\]
then the following statements hold:
\begin{enumerate}
\item For each $s$ the number $\alpha_s$ is the optimal value of a semidefinite programming problem.
\item The equality $\lim_{s\rightarrow\infty}\alpha_s = \alpha$ holds where $\alpha$ is the optimal value of problem~(\ref{prob: BLASSO2}). 
\end{enumerate}
\end{theorem}

In Section~\ref{Sec: numerics} we use Theorem~\ref{thm: hierarchy} for carrying out BLASSO minimization to recover discrete measures and show that we obtain good approximations in dimensions one and two. Our Julia implementation is also made publically available for the community (see Section~\ref{Sec: numerics}).

To conclude this introduction we propose a new application of super-resolution for finding good approximate discretizations of general probability measures on $K$ in the following sense: 

\begin{definition} A $(\delta,k)$-summary of a (not necessarily discrete) positive measure $\mu$ on $K$ with respect to $\phi_1,\dots, \phi_T$ is a positive measure $\Delta$ with at most $k$-atoms for which the following inequality holds
\[\left\|\left(\int_K\phi_i d\mu-\int_K\phi_i d\Delta\right)_{i=1,\dots,T}\right\|_2\leq \delta\] 
\end{definition}
We will assume we know the exact values of the moments of a measure $\mu$ on $K$ and that we would like to find a $(\delta,k)$ summary (for given $\delta$ and $k$). The following Theorem shows that if such a summary exists then it is possible to use super-resolution to approximate it. 

\begin{theorem}\label{thm: summary} Suppose there exists a $(\delta, k)$ summary of $\mu$ supported on a set $X$ and let $\hat{\Delta}$ be a discrete minimizer of the problem

\begin{equation}
\label{prob: Summary}
\min_{\nu\in \mathcal{S}(K)} \|\nu\|_{\rm TV} \text{ : } \left\|\left(\int_K\phi_i d\nu-\int_K\phi_i d\mu\right)_{i=0,\dots,T}\right\|_2\leq \delta.
\end{equation}
If $d$ is sufficiently large then the conclusions of Theorem~\ref{thm: ApproxMain} hold for $\hat\Delta$.
\end{theorem}

Based on the previous Theorem we propose taking the $k$ largest coefficients of a discrete minimizer $\hat{\Delta}$ of~(\ref{prob: Summary}), if such a minimizer exists, as a procedure for summarization. In Section~\ref{Sec: numerics} we present numerical examples of summarization of some measures in dimensions one and two. Our examples in dimension one show that the summarization procedure recovers good approximations of the Gauss-Chebyshev quadrature rule and suggests ways to generalize it to higher dimensions.

{\bf Acknowledgements.} We wish to thank Greg Blekherman, Fabrice Gamboa and Yohann De Castro for very useful conversations during the completion of this project. M. Junca was partially supported by the FAPA funds from Universidad de los Andes. M Velasco was partially supported by Facultad de Ciencias Uniandes grant INV-2018-50-1392. M. Junca, H. Garc\'ia and M. Velasco were partially supported by Colciencias ECOS Nord Colombia-France cooperation Grant EXT-2018-58-1548 {\it Problemas de momentos en control y optimizaci\'on}.

\section{Preliminaries}

\subsection{Representability via discrete measures} \label{Sec: TruncatedMoments}
Let $K\subseteq \RR^n$ be a compact set and let $V$ be a finite-dimensional vector subspace of the space $C(K)$ of continuous real-valued functions on $K$. By a {\it measure} on $K$ we will always mean a positive (and not a complex) measure. We will use the term {\it signed measure} to refer to measures which are real-valued but not necessarily positive. By a positive discrete measure on $K$ we mean a conic combination of Dirac delta measures supported at points of $K$. If $\nu$ is a finite Borel measure on $K$ let $L_{\nu}: C(X)\rightarrow \RR$  be the map given by $L_{\nu}(f):=\int_K f d\nu$. We say that an operator $L:V\rightarrow \RR$ is representable by a measure if there exists a finite Borel measure $\nu$ such that $L(f)=L_{\nu}(f)$ for every $f\in V$. For a vector space $V$ we denote its dual space $V^*:=\{L:V\rightarrow \RR\text{ linear}\}$ and for a cone $C\subseteq V$ let its dual $C^*:=\{\ell\in V^*: \ell(C)\geq 0\}$.

The following Lemma, due to Blekherman and Fialkow~\cite{GBF}, explains the key role played by discrete measures in truncated moment problems. It is a generalization of results of Tchakaloff~\cite{Tchakaloff} and Putinar~\cite{Putinar2}. We include a proof for the reader's benefit.

\begin{lemma} \label{Lem: basicTM} If the functions in $V$ have no common zeroes on $K$ then every linear operator $L: V\rightarrow \RR$ representable by a positive measure is representable by a positive discrete measure with at most $\dim(V)+1$ atoms.
\end{lemma}
\begin{proof} Let $P\subseteq V$ be the closed convex cone of functions in $V$ which are nonnegative at all points of $K$. It is immediate that $P={\rm Conv}(L_{\delta_x}:x\in K)^*$ where $*$ denotes the dual cone. By the bi-duality Theorem from convex geometry we conclude that $P^*= \overline{{\rm Conv}(L_{\delta_{x}}:x\in K)}$. 
Now consider the map $\phi: K\rightarrow V^*$ sending a point $x$ to the restriction of $L_{\delta_x}$ (i.e. to the evaluation at $x$). This map is continuous and therefore $S:=\phi(K)$ is a compact set. Since the functions in $V$ have no points in common the convex hull of $S$ does not contain zero and therefore the cone of discrete measures ${\rm Conv}(L_{\delta_x}:x\in K)$ is closed in $V^*$. Let $\mathcal{M}(V)\subseteq V^ *$ be the cone of operators representable by a finite borel measure. Since 
${\rm Conv}(L_{\delta_x}:x\in K) \subseteq \mathcal{M}(V)\subseteq P^*$ we conclude that $\mathcal{M}(V)$ equals the cone of discrete measures as claimed. The bound on the number of atoms follows from Caratheodory's Theorem~\cite{BarvinokConvex}.
\end{proof}

\subsection{Ideals and coordinate rings of points in projective space}
Suppose $X\subseteq \RR^n$ is a finite set of points of size $k$. To be able to make arguments with graded rings we will embed $X$ in the real projective space $\PP^n$. For basic background on graded rings and projective space the reader should refer to~\cite[Chapter 1,2,8]{CLO}.

We endow $\PP^n$ with homogeneous coordinates $[X_0:\dots: X_n]$ and identify $\RR^n$ with the open subset of $\PP^n$ where $X_0\neq 0$ via the map $\phi(x_1,\dots, x_n)=[1:x_1:\dots:x_n]$. We identify $X$ with its image under $\phi$ and define the homogeneous coordinate ring of $X$ as $A:=\RR[X_0,\dots, X_n]/I(X)$ where $I(X)$ is the ideal generated by all homogeneous polynomials vanishing at all points of $X$. Since $X\subseteq \PP^n$, the ring $A$ is standardly graded (i.e. $A_t:=\RR[X_0,\dots, X_n]_t/I(X)_t$) and is generated, as an algebra over $\RR$, by elements of degree one. Denote by $HF(A,t):=\dim_{\RR} A_t$ the Hilbert function of $A$. The following Lemma summarizes some key basic facts about the homogeneous coordinate ring of a set of $k$ points in $\PP^n$. These are well known classical results in algebraic geometry for which we provide a self-contained elementary proof (see~\cite[Chapter 3]{Eis} for further background on ideals of points in projective space).

\begin{lemma}\label{lem: points} The following statements hold:
\begin{enumerate}
\item The Hilbert function of $A$ is strictly increasing until it attains the value $k$ and then becomes constant.
\item The equality $i(X)=\min\{t: HF(A,t)=k\}$ holds and $i(X)\leq k-1$.
\item The degree of every minimal homogeneous generator of $I(X)$ is bounded above by $\alpha(X):=i(X)+1$.
\end{enumerate}
\end{lemma}
\begin{proof} $(1)$ Let $\ell\in A$ be a linear form which does not vanish at any point of $X$ (for instance $X_0$).  If $F\in A$ satisfies $\ell F=0$ then $F$ must vanish at all points of $X$ and therefore $F=0$ in $A$. We conclude that multiplication by $\ell$, $m_{\ell}: A_t\rightarrow A_{t+1}$ is injective for every $t\geq 0$ proving that $HF(A,t)$ is non-decreasing. Let $B:=A/(\ell)$ and note that for every $t$ we have $B_t=0$ if and only if $HF(A,t)=HF(A,t-1)$. Since $B$ is generated in degree one the equality $B_t=0$ for some $t$ implies that $B_{r}=0$ for all $r\geq t$. We conclude that if $t$ satisfies $HF(A,t)=HF(A,t-1)$ then the Hilbert function becomes constant after $t$, proving $(1)$. $(2)$ For any $t\in \mathbb{N}$ consider the linear map $\phi^*: A_t \rightarrow {\rm Fun}(X,\RR)$ which maps $F(X_0,\dots, X_n)$ to the polynomial function $F(1,x_1,\dots,x_n)$ of degree at most $t$ on $X$. This map is always injective and is therefore surjective whenever the dimension of $A_t$ equals the dimension $k$ of the space of all real-valued functions on $X$. To prove the inequality note that $HF(A,0)=1$ and that it increases strictly at every stage so $HF(A,k-1)\geq k$ so $i(X)\leq k-1$.

$(3)$ Let $J$ be the ideal generated by $I(X)_{\leq \alpha(X)}$ and let $S:=\RR[X_0,\dots, X_n]$. Since $J\subseteq I(X)$ there is a surjective homomorphism $A':=S/J\rightarrow A$ and we will show that it is an isomorphism by proving that $\dim A'_t=\dim A_t$ for all $t$. Define the quotient ring $Q:=A'/(\ell)$ and note that it satisfies $Q_j=B_j$ for $j\leq \alpha(X)$ and in particular $Q_{\alpha(X)}=0$. Since $Q$ is generated in degree one this implies that $Q_q=0$ for all $q\geq \alpha(X)$ and therefore multiplication by $\ell$ is surjective on $A'$ in all components $t\geq \alpha(X)$. We conclude that $\dim A'_t\geq \dim A'_{t+1}$ for $t\geq \alpha(X)$ and in particular $k\geq \dim A'_{s}$ for $s\geq \alpha(X)$. By surjectivity of $A'\rightarrow A$ we know that $\dim A'_t\geq \dim A_t=k$ for $t\geq \alpha(X)$. Putting both inequalities together we conclude that $\dim A'_t=\dim A_t$ for all $t$ as claimed. 

\end{proof}

\begin{remark} The number $\alpha(X)$ is the Castelnuovo-Mumford regularity of $X$, the key measure of the (cohomological) complexity of algebraic varieties~\cite{Mumford} (see~\cite[Chapter 4]{Eis} for details).
\end{remark}

\subsection{Generic points}\label{sec: generic}
A property of $k$-tuples of points $(p_1,\dots, p_k)\in (\PP^n)^k$ holds generically if the locus of points $(p_1,\dots, p_k)$ which satisfy it contains a nonempty Zariski open set. Equivalently, the set of points where the property fails is contained in a proper Zariski closed subset of $(\PP^n)^k$ (i.e. one defined by homogeneous polynomial equations). Following common terminology we say that {\it a generic set of points $X$ of size $k$ satisfies a property $Q$} to mean that property $Q$ holds generically. If $p_1,\dots, p_k$ are an independent sample of points in $\RR^n$ sampled from a distribution which has a density with respect to the Lebesgue measure then $p_1,\dots, p_k$ satisfies every generic property with probability one (because every proper Zariski closed set has empty interior and in particular null Lebesgue measure). Understanding generic properties should therefore be of much interest for applications since those are the only ones that arise for ''randomly chosen" or ''noisy" sets of points.

\section{A basic uniqueness result for exact super-resolution.}

\label{Sec: Uniqueness}
\begin{proof}[Proof of Theorem~\ref{thm: ExactUniqueness}] Assume $d\geq \max\left(2g(X), i(X)+1\right)$ and let $\mu:=\sum_{x\in X} c_x\delta_x$ for some real coefficients $c_x\geq 0$ and let $h_1,\dots, h_k$ be a set of generators of the ideal $I(X)$ of polynomials vanishing on $X$. Define $H:=\sum h_i^2$ and $M:=\sup_{x\in K} H(x)$. By our assumption on $d$ the polynomial $P:=1-\frac{H}{M}$ belongs to $V_{\leq d}$. By construction $P$ is a dual certificate in the sense of Cand\'es, Romberg and Tao~\cite{CRT}, this means that $\|P\|_{\infty}= 1$ on $K$ and that $P(z)=1$ if and only if $z\in X$. 
If $\Delta$ is a feasible solution of~(\ref{prob: TV}) then 
\[\|\mu\|_{\rm TV}= \mu(X)=\int_K Pd\mu =\int_K P d \Delta \leq \|\Delta\|_{\rm TV}\] 
and therefore any optimal solution $\Delta$ of~(\ref{prob: TV}) satisfies $\|\mu\|_{\rm TV}=\|\Delta\|_{\rm TV}$. For $\Delta$ an optimal solution of~(\ref{prob: TV}) we write $\Delta=\Delta_X+\Delta_X^{\perp}$ where $\Delta_X$ is supported on $X$ and $\Delta_X^{\perp}$ in $K\setminus X$. Since $|P(z)|<1$ outside $X$ we conclude that $\int_KPd\Delta_X^{\perp}<\|\Delta_X^{\perp}\|_{\rm TV}$ if $\Delta_X^{\perp}\neq 0$. It follows that 
\[\|\Delta\|_{\rm TV}=\int_K Pd\Delta < \|\Delta_X\|_{\rm TV} + \|\Delta_{X}^{\perp}\|_{\rm TV}=\|\Delta\|_{\rm TV}\]
a contradiction so $\Delta_X^{\perp}=0$ and every minimizer $\Delta$ is supported on $X$.
Since $d\geq i(X)$, there exists for each point $x\in X$ a polynomial $q_x$ in $V_{\leq d}$ which takes value one in $x$ and zero at all other points of $X$. Since $\int_Kq_xd\mu=\int_Kq_xd\Delta$ we conclude that $\Delta=\mu$ proving uniqueness. We conclude that $d(X)\leq d$ proving the inequality in part $(1)$. Furthermore, by Lemma~\ref{lem: points} part $(3)$ we know that $\alpha(X)=1+i(X)$ satisfies $g(X)\leq \alpha(X)$ and therefore $\max(2g(X),i(X))\leq 2\alpha(X)$. If $X$ is not contained in any hyperplane then $\dim(A_1)=n+1$ and therefore by Lemma~\ref{lem: points} part $(1)$ $\dim(A_t)\geq n+t$ for all $1\leq t\leq i(X)$ and we conclude that $i(X)\leq k-n$ so $\alpha(X)\leq k-n+1$, proving the claim.
$(2)$ By strong duality, uniqueness implies that there is a polynomial $P$ of degree $d(X)$ which serves as a dual certificate. It follows that $H:=1-P$ is nonnegative in $K$ and has value zero at the points of $X$. We conclude that all points of $X$ are local minima and, since $X\subseteq K^{\circ}$,  critical points for $H$. As a result, the hypersurface defined by $H$ is singular at all points of $X$. We conclude that $d(X)\geq   \ell(X)$ as claimed.\end{proof}

\begin{remark}  If $X\subseteq [a,b] \subseteq \RR$ consists of $k$ points interior to the interval then it is immediate that $g(X)=k$ and $i(X)=k+1$ so our Theorem implies uniqueness for $d\geq 2k$, giving another proof of~\cite[Proposition 2.3]{DeCastroGamboa}. This upper bound is sharp since it agrees with the lower bound $\ell(X)=2k$. 
\end{remark}

As mentioned in the Introduction, the previous Theorem highlights the enormous difference between superresolution in one and in more dimensions. Whereas in one dimension the uniqueness degree depends only on the cardinality of the set of points, in higher-dimension this degree is determined by the structure of the ideal of the set of points. 
The following two examples show that Theorem~\ref{thm: ExactUniqueness} is sharp (i.e. that the inequalities do become equalities for some sets of points) and that there are sets of points in the plane of {\it the same cardinality} for which $d(X)$ has very different behaviors (see Remark~\ref{rmk: 10pts}).

\begin{example}
For a positive integer $d$ let $X\subseteq \RR^2$ be the set of $2d$ points defined by a non-singular quadric and a generic form of degree $d$ and let $K$ be any compact set which contains $X$ in its interior. A nonsingular plane quadric can be parametrized via a map $\phi: \RR\rightarrow \RR ^2$ with quadratic monomials. It follows that if $F$ is a form of degree $\ell$ which
is singular at the points of $X$ then $F\circ \phi$ is a univariate polynomial of degree $2\ell$ which is singular at $2d$ points and therefore $2\ell\geq 4d$. We conclude from Theorem~\ref{thm: ExactUniqueness} part $(2)$ that $d(X)\geq \ 2d$. Since $g(X)=d$ and $i(X)=d$, Theorem~\ref{thm: ExactUniqueness} part $(1)$ implies that $d(X)\leq 2d$. We conclude that $d(X)=2d$ so Theorem~\ref{thm: ExactUniqueness} is sharp for infinitely many point sets in the plane.
\end{example}

\begin{example}\label{Exm: 10pts} By the genus formula~\cite[pg. 53]{ACGH} a plane curve of degree $d$ can have at most $\binom{d-1}{2}$ singular points. For plane curves of degree $6$ it is possible~\cite[Proposition 5.7]{BSV2} to construct curves $C$ where this maximum is achieved at real points and furthermore those are the only real points of $C$. It follows that there is a polynomial $P$ which defines $F$ which is nonnegative in $\RR^2$ and whose only real zeroes are the $10$ nodes. Let $X$ be the set of $10$ nodes and let $K$ be a compact set which contains $X$ in its interior. If $M$ denotes the maximum value of $P$ in $K$ then the polynomial $1-\frac{P}{M}$ is a dual certificate as in the proof of Theorem~\ref{thm: ExactUniqueness} and therefore $d(X)\leq 6$. The genus formula guarantees that no form of degree $\leq 5$ can be singular at all points of $X$ so we conclude from Theorem~\ref{thm: ExactUniqueness} part $(2)$ that $d(X)\geq 6$ and therefore $d(X)=6$. 
\end{example}

\begin{remark}\label{rmk: 10pts} It follows from the previous two examples that there are sets of points $X$ of cardinality $10$ with $d(X)=10$ and with $d(X)=6$ depending on the structure of their ideal of definition.
\end{remark}

\begin{proof}[Proof of Theorem~\ref{thm: GenericExactUniqueness}] $(1)$ The Hilbert function of the homogeneous coordinate ring $A$ of a generic set of $k$ points in $\PP^n$ is given by
\[ HF(A,t)=\min\left(\binom{n+t}{t},k\right)\]
which coincides with $k$ for the smallest $e$ with $\binom{n+e}{e}\geq k$. We conclude that $i(X)=e$. Moreover, 
by Lemma~\ref{lem: points} we know that $\alpha(X)= i(X)+1$ and $\max(2g(X),i(X))\leq 2\alpha(X)$ proving the claim from Theorem~\ref{thm: ExactUniqueness} part $(1)$.

$(2)$ Since $X\subseteq \RR^n$, vanishing with multiplicty at least two at a point of $X$ imposes $n+1$ linear conditions (vanishing at the point and vanishing of the $n$-partial derivatives at the point). Since the points of $X$ are generic, such linear conditions are independent obtaining $k(n+1)$ independent conditions. It follows that there does not exist a polynomial $F$ which is singular at all points of $X$ whenever $\binom{n+d}{d}< k(n+1)$ proving the inequality by Theorem~\ref{thm: ExactUniqueness} part $(2)$.
\end{proof}

\begin{remark} The maximum in the quantity $\max(2g(X),i(X))$ of Theorem~\ref{thm: ExactUniqueness} can be achieved in either side as the following examples show. If $X$ is a complete intersection of $n$ quadrics in $\PP^n$ then $g(X)=2$ and $i(X)=n-1$ so $i(X)>2g(X)$ for $n\geq 5$. If $d>0$ and $X$ is a generic set of $\binom{d+n}{n}$ points in $\PP^n$ then $g(X)=d+1$ and $i(X)=d$ so $2g(X)>i(X)$.
\end{remark}

We believe that the true value of the uniqueness degree for generic sets of points is when it is equal to the lower bound in Theorem~\ref{thm: GenericExactUniqueness}. We think this is the case because the space of measures supported on $k$ points in $\RR^n$ has dimension $k(n+1)$ ($n$ coefficients for specifying the location of each point and one more for specifying the accompanying coefficient). 
As a result, if the lower bound from Theorem~\ref{thm: GenericExactUniqueness} is satisfied then we have enough linear measurements to encode the space of measures (at least locally) and we believe this should be enough for convex optimization to be able to recover a point measure uniquely.  

\begin{remark} The degrees of all minimal generators, and more generally the structure of the minimal free resolutions of ideals of points in $\PP^2$ are well understood (See~\cite[Chapter 3]{Eis} for details). By contrast the minimal free resolution of even generic sets of points $s$ in $\PP^n$ for $n>2$ is widely open. The conjectural answer suggested by Lorenzini~\cite{Lorenzini} was later disproved in celebrated work by Eisenbud and Popescu~\cite{EisenbudPopescu}.
\end{remark}

\subsection{A moments relaxation}\label{Sec: moments} Mirroring the proof of Theorem~\ref{thm: ExactUniqueness} one can use the following problem of moments reconstruction procedure for recovering $\mu$ given its moments operator $L:V_{\leq 2d}\rightarrow \RR$ with $L(g):=\int_K gd\mu$ on polynomials of degree at most $2d$. 
\begin{enumerate}
\item Finding the {\it support} of $\mu$ by constructing a minimizer $H^*$ of the optimization problem $\min_H L(H)$ where $H$ runs over the sums-of-squares of elements of $V_{\leq d}$. More explicitly if $\vec{\phi}$ is a basis for $V_{\leq d}$ then we find $H^*$ by solving the semidefinite programming problem:
\[ \min L\left(\vec{\phi}^tA\vec{\phi}\right) \text{ s.t. $A\succeq 0$ \text{ and } ${\rm tr}(A)=1.$ }\]
and find the support of $\mu$ by finding the zeroes of $H^*$ in $K$ (the trace restriction prevents the trivial solution).

\item Finding the {\it coefficients} of $\mu$ by linear algebra. If $z_1,\dots, z_k$ are the zeroes of $H^*$ we find the coefficients $c_1,\dots, c_k$ by solving the linear equations $\sum_{i=1}^k c_i f(z_i)=L(f)$ for $f\in V$. 
\end{enumerate}
Theorem~\ref{thm: ExactUniqueness} guarantees that the procedure works for $d$ above the upper bound and Theorem~\ref{thm: GenericExactUniqueness} gives an explicit upper bound for measures supported on generic points. We finish the Section with two remarks about the above procedure:
\begin{enumerate}
\item Assume $H^*$ is any minimizer in the relative interior of the face $L(H)=0$ of convex cone $Q$ of the sums-of-squares of elements in $V_{\leq d}$ with $d\geq g(X)$. Then $H^*$ must have $X$ as its only real zeroes since otherwise evaluation at any additional zero would define a proper face of $Q$ containing an interior point and hence all of $Q$. In particular the kernel of this evaluation would contain the dual certificate constructed in the proof of Theorem~\ref{thm: ExactUniqueness} all of whose real zeroes lie on $X$ deriving a contradiction. As a result, interior point numerical methods for solving the SDP would produce optima $H^*$ with $X$ as its set of zeroes, as can be seen in our numerical examples in Section~\ref{Sec: numerics}.

\item ({\it Sharpness}) If $X$ is a generic set of $k=\binom{e+n}{n}$ points in $\PP^n$ then $g(X)=i(X)=e+1$ and Theorem~\ref{thm: GenericExactUniqueness} shows that there is unique recovery when $d\geq 2(e+1)$. We claim that, {\it if the recovery is carried out with the sum-of-squares procedure above} then this bound is sharp in the sense that the recovery would fail for $d<2(e+1)$. The reason is that every sum of squares $H=\sum P_i^2$ which vanishes at the points would have summands $P_i$ of degree less than $e+1$ and therefore be identically zero on $X$ because $I(X)$ contains no forms of degree less than $e+1$. Note that this does not preclude the existence of lower degree certificates that are not sums-of-squares as the one appearing in Example~\ref{Exm: 10pts}.
\end{enumerate}

\section{Approximate recovery}

In this section we focus on the problem of approximate recovery. The following key property was proposed by Azais, De Castro and Gamboa as central for BLASSO quantitative localization results. We modify their definition slightly since our interest is the recovery of positive measures and not of signed measures.
It is well known in the super-resolution community that positivity of the measure is a strong assumption~\cite{MC, ETTTT, DDP}. We believe that it is a reasonable starting point for trying to extend the super-resolution results to the more complicated higher-dimensional setting.

\begin{definition}(Quadratic isolation condition)\cite[Definition 2.2]{AzaisDeCastroGamboa} A finite set $X\subseteq K$ satisfies a quadratic isolation condition with parameters $C_a>0$ and $0<C_b<1$ respect to $V$ if there exists $P\in V$ satisfying $\|P\|_{\infty}\leq 1$ on $K$, $P\equiv 1$ on $X$ and such that the following inequality holds 
\[\forall z\in K\left( P(z)\leq \max\left\{1-C_ad(z,X)^2, 1-C_b\right\}\right)\]
\end{definition}
In that case we say that $P$ is a witness for a QIC condition on $X$.

\begin{lemma}\label{lem: QIC} If $d\geq 2g(X)$ then $X$ satisfies a quadratic isolation condition on $V_{\leq d}$. 
\end{lemma}
\begin{proof} Let $f_1,\dots, f_s$ be a set of minimal generators of the ideal $I(X)$ of polynomials vanishing on $X$ and define $H:=\sum f_i^2$ and $M:=\sup_{x\in K} H(x)$. By our assumption on $d$ the polynomial $P:=1-\frac{H}{M}$ is nonnegative, belongs to $V_{\leq d}$ and is identical to one on $X$.
Since $f_1,\dots, f_s$ are generators of the ideal $I(X)$ and $X$ is a nonsingular variety the differential of the map $\mathcal{H}:\RR^n\rightarrow \RR^s$ given by $\mathcal{H}(x)=(f_1(x),\dots, f_s(x))$ has trivial kernel at every $x\in X$. As a result, the Hessian at $x\in X$ of the polynomial $H$ is positive definite and in particular there exist positive real numbers $\eta_x$ and $(C_a)_x$ such that $1-P(z)\geq (C_a)_x\|z-x\|^2$ for $z$ with $\|z-x\|\leq \eta_x$. Define $\delta=\min_{x\in X}\eta_x$, $C_a:=\min_{x\in X} (C_a)_x$ and let $C_b=\inf_{z: d(z,X)\geq \frac{\delta}{2}}\left(1-P(z)\right)$. We conclude that $X$ satisfies a quadratic isolation condition with parameters $C_a$ and $C_b$ with $0<C_b<1$ and $C_a>0$.  
\end{proof}

The following Lemma, of interest in its own right, extracts the essence of~\cite[Theorem 2.1]{AzaisDeCastroGamboa}. It is the key technical tool for converting QIC witnesses into quantitative localization guarantees. Henceforth we will assume that $\phi_0,\dots, \phi_T$ are a basis for our space of functions which is orthonormal with respect to some probability measure on $K$ 

\begin{lemma}\label{lem: key} Let $P$ be a polynomial with $\|P\|_{\infty}\leq 1$ on $K$. The following statements hold:
\begin{enumerate}
\item If $\hat{\Delta}$ is feasible for~(\ref{prob: BLASSO}) then the following inequality holds
\[\left|\int_K Pd\hat{\Delta}-\int_K Pd\mu\right|\leq 2\delta.\]

\item If $\hat{\Delta}$ is a minimizer of~(\ref{prob: BLASSO}) then the following inequality holds
\[ 0\leq \|\hat{\Delta}\|_{\rm TV}-\int_K Pd\hat{\Delta} \leq 2\delta.\]
\end{enumerate}
\end{lemma}
\begin{proof}  For $(1)$ suppose $P=\sum_i a_i \phi_i$ and note that
\[\left|\int_K Pd\hat{\Delta}-\int_K Pd\mu\right| \leq \sum_i |a_i|\left|\int_K \phi_i d\hat{\Delta}-\int_K \phi_i d\mu\right|.\]
By the Cauchy-Schwartz inequality this quantity is bounded above by
\[\leq \|a\|_2\left\|\left(\int_K \phi_i d\hat{\Delta}-\int_K \phi_i d\mu\right)_i\right\|_2\leq \|a\|_2 2\delta\]
where the last inequality holds because, by feasibilty of $\hat{\Delta}$
\[\left\|\left(\int_K \phi_i d\hat{\Delta}-\int_K \phi_i d\mu\right)_i\right\|_2\leq \left\|\left(\int_K \phi_i d\hat{\Delta}-y_i\right)\right\|_2+ \left\|\left(y_i -\int_K \phi_i d\mu\right)_i\right\|_2\leq 2\delta.\]
Next, using the fact that the $\phi_i$ are orthonormal with respect to some probability measure $\tau$, we conclude by Parseval's equality that $\|a\|_2=\|P\|_{L^2(\tau)}$ and this quantity is at most one since $\|P\|_{\infty}\leq 1$ proving the claim.

$(2)$ Since the measure $\mu$ is feasible for~(\ref{prob: BLASSO}) we know that $\|\mu\|_{\rm TV}\geq \|\hat{\Delta}\|_{\rm TV}$. Since $P$ satisfies $\|P\|_{\infty}\leq 1$ on $K$ we know that 
\[\|\hat{\Delta}\|_{\rm TV}\geq \int_KPd\hat\Delta = \sum a_i \int_K\phi_id\hat\Delta = \sum a_i\left(r_i+\int_K\phi_i d\mu +\epsilon_i\right)\]
where $r_i:=\int_K\phi_id\hat\Delta-y_i$ and $y_i=\int_K\phi_id\mu+\epsilon_i$. Since $\sum a_i\int_K\phi_id\mu =\int_KPd\mu=\|\mu\|_{\rm TV}$ we conclude that
\[\|\mu\|_{\rm TV}+\sum a_i(r_i+\epsilon_i)\leq \int_K Pd\hat{\Delta}\leq \|\hat{\Delta}\|_{\rm TV}\leq\|\mu\|_{\rm TV}\]
So the difference between the last and first terms is an upper bound for the difference between the interior terms yielding

\[0\leq \|\hat{\Delta}\|_{\rm TV}-\int_K Pd\hat{\Delta}\leq \left|\sum a_i(r_i+\epsilon_i)\right|\leq \|a\|_2\left(\|r\|_2+\|\epsilon\|_2\right)\leq 2\delta\|a\|_2\]
where the last two inequalities follow from the Cauchy-Schwartz and triangle inequalities. Using Parseval's equality we obtain the desired conclusion. \end{proof}

The following Theorem explains how a Quadratic Isolation condition leads to a quantitative localization guarantee. In words it says that large recovered spikes cannot occur far from the support of the original measure.

\begin{theorem}\label{thm: Approx} Let $\mu$ be any positive discrete measure supported on a finite set $X$ and let $\hat{\Delta}$ be a discrete minimizer of~(\ref{prob: BLASSO}) with $V=V_{\leq d}$, $y_i:=\int\phi_id\mu +\epsilon_i$ and $\|(\epsilon_i)_i\|_2\leq \delta$. If $d\geq 2g(X)$ then there exist constants $C_a>0$ and $0<C_b<1$ depending only on $X$ such that if $c_0:=\sqrt{\frac{C_b}{C_a}}$ then the following statements hold:
\begin{enumerate}
\item If $z\in K$ is such that $\hat{\Delta}(z)>\frac{2\delta}{C_b}$ then $d(X,z)\leq c_0$.
\item The following inequalities hold:
\[
\sum_{z\in N\left(X,c_0\right), \hat\Delta(z)>0} \hat{\Delta}(z) d(X,z)^2\leq \frac{2\delta}{C_a}\]
\[\sum_{z\in F\left(X,c_0\right), \hat\Delta(z)>0} \hat{\Delta}(z) \leq \frac{2\delta}{C_b}\]
\[\sum_{z: \hat\Delta(z)<0}|\hat{\Delta}(z)|\leq 2\delta\]
\end{enumerate}
\end{theorem}

\begin{proof} By Lemma~\ref{lem: QIC} the set $X$ satisfies a quadratic isolation condition with parameters $C_a>0$ and $0<C_b<1$. If $P$ is the witness constructed in Lemma~\ref{lem: QIC} then Lemma~\ref{lem: key} implies that any minimizer $\hat\Delta$ of problem~(\ref{prob: BLASSO}) satisfies 
\[0\leq \|\hat\Delta\|_{\rm TV}-\int_KPd\hat{\Delta}\leq 2\delta.\]

Assuming $\hat\Delta=\sum_{z\in K} \hat\Delta(z)\delta_{z}$ we estimate the quantity in the middle using the fact that, by the QIC, the inequality
\[1-P(z)\geq \min\{C_ad(X,z)^2,C_b\}\]
holds. Separating the coefficients of $\hat{\Delta}$ into three sets: negative coefficients, and two sets of positive coefficients according to which of the two terms achieves the minimum in $\min\{C_ad(z,X)^2, C_b\}$ we obtain, since $P\geq 0$, the inequality

\begin{small}
\[2\delta\geq \sum_{\hat\Delta(z)<0} |\hat\Delta(z)| + \sum_{\hat\Delta(z)>0,d(z,X)^2\leq \frac{C_b}{C_a}} \hat\Delta(z) C_a d(z,X)^2 +  \sum_{\hat\Delta(z)>0,d(z,X)^2>\frac{C_b}{C_a}}\hat\Delta(z) C_b\]
\end{small}
holds, from which the three inequalities in part $(2)$ of the Theorem follow immediately. 
\end{proof}

\subsection{Quantitative results on approximate recovery.}
\label{Sec: QuantitativeApprox}
Theorem~\ref{thm: Approx} shows that a witness of a quadratic isolation condition gives quantitative bounds for the approximation quality of solutions of the super-resolution problem~\eqref{prob: BLASSO}. In this section we study how the constants $C_a$ and $C_b$ depend on the geometry of the underlying support set. Our main result is the proof of Theorem~\ref{thm: ApproxMain} which answers the key question of how the approximation constants vary as functions of degree and error-size. In Example~\ref{Exo: Grids} we specialize our results to the case of grids in $\RR^n$ connecting the above constants with the distances between pairs of nearby points. We begin by strengthening Lemma~\ref{lem: QIC}.

\begin{lemma}\label{lem: genpoints} Let $X\subseteq K\subseteq \RR^n$ be a finite set. Suppose $I(X)=(f_1,\dots, f_s)$ and let $h=(f_1/M_1)^2+\dots + (f_s/M_s)^2$ with $M_i:=\sup_{z\in K} f_i(z)$. There exist positive constants $\eta, D, D_1$ such that:
\begin{enumerate}
\item For all $z\in N(X,\eta)$ $h(z)\geq D d(z,X)^2$,
\item For every $z\in F(X,\eta)$ there exists an index $j(z)\in \{1,\dots, s\}$ such that $|f_j(z)/M_j|\geq D_1$. 
\end{enumerate}
\end{lemma}
\begin{proof} $(1)$ If $x^*\in X$, $z\in \RR^n$ and $j=1,\dots, s$ then $\frac{f_j(z)}{M_j}=\ell_j(z-x^*)+o(\|z-x^*\|^2)$ where $\ell_j(z-x^*)=\langle\nabla \left(\frac{f_j(x^*)}{M_j}\right),z-x^*\rangle$ and therefore there exists $\eta_{x^*}>0$ such that, if $\|z-x^*\|\leq \eta_{x^*}$ then 
\[h(z)=\sum_{j=1}^s \frac{f_j(z)}{M_j}^2\geq \sum \frac{1}{2}\ell_j(z-x^*)^2\]
and such that $x^*$ is the closest point of $X$ to any $z$ with $\|z-x^*\|\leq \eta_{x^*}$. Since $f_1,\dots, f_s$ define the non-singular variety $X$ the linear terms $\ell_j$ have no common zeroes and in particular the quadratic form in the right hand side is positive definite. It follows that there exists a constant $A_{x^*}>0$ such that 
\[h(z)\geq  \sum \frac{1}{2}\ell_j(z-x^*)^2\geq A_{x^*}\|z-x^*\|^2\]

Choosing $\eta:=\min_{x^*\in X} \eta_{x^*}$ and $D=\min_{x^*\in X} A_{x^*}$ we conclude that, whenever $d(z,X)\leq \eta$ the inequality $h(z)\geq D d(z,X)^2$ holds.  $(2)$ The set $F(X,\eta)$ is compact and therefore the continuous function $h(z)$ achieves a minimum value $\gamma>0$ on it. We conclude that for every $z\in F(X,\eta)$ there exists an index $j$, which may depend on $z$, such that $|f_j(z)/M_j|\geq \sqrt{\gamma/s}$. Letting $D:=\sqrt{\gamma/s}$ proves the claim.
\end{proof}

For the following Lemma we will use some basic properties of Chebyshev polynomials. Recall that the $n$-th Chebyshev polynomial $T_n(x)$ is the unique univariate polynomial of degree $n$ which satisfies the equality $T_n(x)=\cos(n\arccos(x))$ for $x\in [-1,1]$ or equivalently 
\[ T_n(\cos(\theta)) = \cos(n\theta)={\rm Re}(e^{in\theta})\text{ for $\theta\in [0,\pi]$}.\]
The Chebyshev polynomials have many remarkable properties, for instance they are orthonormal in $[-1,1]$ with respect to the weight function $(1-x^2)^{-\frac{1}{2}}$. We will use the Chebyshev polynomials to construct a special approximation of the Dirac $\delta$ distribution centered at the origin in $\mathbb{R}$. We will then use these univariate approximations to build witnesses of the Quadratic isolation condition for finite sets of $\RR^n$. The technical tools are summarized in the following elementary Lemma whose content is visualized in Figure~\ref{Fig: Sombrero} below.

\begin{figure}[t]
\centering
\includegraphics[width=0.7\textwidth]{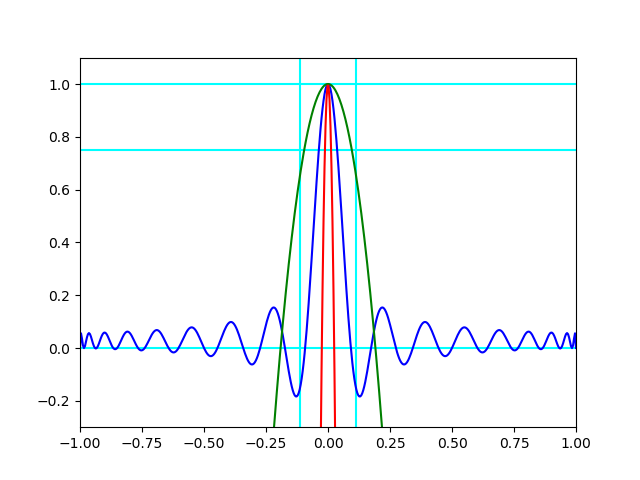}
\caption{The polynomial $H_{18}(x)$ (blue) and its lower (red) and upper (green) bounds}
\label{Fig: Sombrero} 
\end{figure}

\begin{lemma} \label{lem: sombrero}
For an even positive integer $m$ define \[H_m(x)=\frac{1}{m}\left(\sum_{k=0}^{m-1} (-1)^k T_{2k}(x)\right).\]
The following statements hold
\begin{enumerate}
\item $H_m(x)$ is a polynomial of degree $2(m-1)$.
\item $|H_m(x)|\leq 1$ for $x\in [-1,1]$ and the maximum value of one is achieved only at $x=0$. 
\item The following equality holds for $z\in \left[-\frac{\pi}{2},\frac{\pi}{2}\right]$
\[ H_m(\sin(z)) = \frac{1}{2m} + \frac{\cos(2(m-1)z)-\cos(2mz)}{2m(1-\cos(2z))}\]
and in particular $H_m$ is an even function.
\item The following inequalities hold for $-1\leq x\leq 1$ (equiv. for $-\pi/2\leq z\leq \pi/2$):
\begin{enumerate}
\item If $m\geq 2$ then $|H_m(\sin(z))|\leq \frac{3}{4}$ whenever $|z|\geq \frac{2}{m}$.
\item If $|z|\leq \frac{2}{m}$ then $|H_m(\sin(z))|\leq 1-\frac{m^2}{3!2}\sin^2(z)$ and therefore \[|H_m(x)|\leq \max\left(3/4, 1-\frac{m^2}{3!2}x^2\right)\text{ for every $x\in [-1,1]$}.\] 
\item For every $x\in [-1,1]$ the following inequality holds
\[H_m(x)\geq 1-\frac{\pi^2}{2}m^2x^2\]

\end{enumerate}
\end{enumerate}
\end{lemma}
\begin{proof} $(1)$ Since $T_k(x)$ is a polynomial of degree $k$ it follows that $H_m(x)$ is a polynomial of degree $2(m-1)$. $(2)$ It is immediate from the trigonometric definition that the Chebyshev polynomials satisfy $|T_{k}(x)|\leq 1$ in $[-1,1]$ and that the maximum value is achieved when $x=\cos(\theta)$ where $\theta$ is a solution of $\cos(k\theta)=1$. It follows that any convex combination of the polynomials $(-1)^kT_{2k}$ takes values in $[-1,1]$ and that the value $1$ is achieved only at the common solutions of $(-1)^kT_{2k}(x)=1$ for $k=0,\dots, m-1$, that is, only when $x=0$. $(3)$ Since $T_m(\cos\theta)$ is the real part of $e^{im\theta}$ we can rewrite $H_m(\cos(\theta))$ as the real part of $Z_m(\theta):=\frac{1}{m}\sum_{k=0}^{m-1} (-e^{i2\theta})^k$. Computing the geometric sum we conclude that 
\[ H_m(\cos(\theta))={\rm Re}(Z_m(\theta))={\rm Re}\left(\frac{1+(-1)^{m-1}e^{i2m\theta}}{m(1+e^{i2\theta})}\right).\]
Computing the real part for $m$ even we conclude that
\[H_m(\cos(\theta)) = \frac{1+\cos(2\theta) -\left(\cos(2m\theta)+\cos(2(m-1)\theta)\right)}{2m(1+\cos(2\theta))}\]
Furthermore, setting $z=\theta-\frac{\pi}{2}$ we obtain the equality
\[H_m(\sin(z))= \frac{1}{2m} + \frac{\cos(2(m-1)z)-\cos(2mz)}{2m(1-\cos(2z))}\]
proving $(3)$. $(4)$ For the inequalities in part $(4)$ we will freely use the fact that the Taylor expansions of order $m$ of $\sin(x)$ and $\cos(x)$ are lower (resp. upper) bounds on them when $m$ is even (resp. odd). For instance, the following inequalities hold
\[ \theta\geq \theta-\frac{\theta^3}{3!} + \frac{\theta^5}{5!} \geq\sin(\theta)\geq \theta-\frac{\theta^3}{3!} + \frac{\theta^5}{5!} -  \frac{\theta^7}{7!}\geq \theta-\frac{\theta^3}{3!}.\]
$(4a)$ Setting $z=\theta-\frac{\pi}{2}$ and using the fact that $|1-e^{i2\theta}|= 2|\sin(\theta)|$ we conclude that the norm of the complex number $Z_m(\theta)$ above is given by
\[\left|Z_m(\theta)\right| = \frac{|\sin(mz)|}{m|\sin(z)|}\]
It follows that if $|z|\geq \frac{2}{m}$ then 
\[|H_m(\sin(z))|\leq \frac{|\sin(mz)|}{m|\sin(z)|}\leq \frac{1}{m\left(\frac{2}{m}-\frac{2^3}{m^33!}\right)}\leq \frac{3}{4}.\]
where the last inequality holds since $m\geq 2$.
$(4b)$ From the Taylor expansion inequalities for cosine above we know that
for $0\leq z\leq mz\leq \pi$ the following inequalities holds
\begin{align*}
\frac{\sin(mz)}{m\sin(z)} \leq \frac{mz-\frac{(mz)^3}{3!}+\frac{(mz)^5}{5!}}{m\left(z-\frac{z^3}{3!}\right)}\leq 1+\frac{1-m^2}{3!}z^2 + \frac{\frac{m^4}{5!}+\frac{1-m^2}{(3!)^2}}{\left(1-\frac{z^2}{3!}\right)}z^4 \leq 1-\frac{m^2}{3!2}z^2\leq 1-\frac{m^2}{3!2}\sin^2(z)
\end{align*}

where the second to last inequality holds because it is equivalent to the inequality $z^2\leq \frac{5(m^2-1)}{m^4}$ which is implied by our assumption that $|z|\leq \frac{2}{m}$ proving $(4b)$. 

For the final inequality $(4c)$ recall that 
\[(-1)^kT_{2k}(\sin(z))=\cos(2kz).\]
It follows that for $k=0,\dots,m-1$ the inequality
\[\cos(2kz)\geq 1-\frac{(2kz)^2}{2}\geq 1-2m^2z^2\]
holds. Using the fact that $z\leq \frac{\pi\sin(z)}{2}$ for $z\in [0,\frac{\pi}{2}]$ we conclude that

\[H_m(x)\geq 1-\frac{\pi^2}{2}m^2x^2\]
for every $x\in [-1,1]$ proving $(4c)$
\end{proof}

\begin{remark} The motivation for defining $H_m(x)$ in the previous Lemma comes from the idea of trying to approximate the Dirac distribution centered at $0$ as a combination of orthogonal functions. Heuristically, if $\delta_0=\sum a_i\phi_i$ then $a_j=\int \phi_j d\delta_0$ by orthogonality of the $\phi_i$  and thus $a_j = \phi_j(0)$ since $\delta_0$ represents the evaluation function at zero. Using easy properties of Chebyshev polynomials these approximations lead to the helper polynomials $H_m(x)$ of the previous Lemma.
\end{remark}

We are now in a position to prove our main result on approximate super-resolution.
 
\begin{proof} [Proof of Theorem~\ref{thm: ApproxMain}] Suppose $I(X)=(f_1,\dots, f_s)$ and let $M_i:=\sup_{x\in K} |f_i|$. Define the polynomial $P_m(z)$ on $K$ as
\[ P_m(z)=\frac{1}{s}\left(H_m\left(\frac{f_1(z)}{M_1}\right)+\dots +H_m\left(\frac{f_s(z)}{M_s}\right)\right).\]
Note that $P$ is a polynomial of degree $2(m-1)g(X)$. We will show that for all sufficiently large even integers $m$ the polynomial $P_m(z)$ is a witness for a quadratic isolation condition on $X$. Crucially the constants $C_a$ and $C_b$ will depend on $m$ allowing us to understand how the approximation quality varies with the degree. The proof proceeds by verifying the following claims:
\begin{enumerate}
\item $\|P_m\|_{\infty}\leq 1$ on $K$. This is because $P_m$ is a convex combination of polynomials $H_m\left(\frac{f_i}{M_i}\right)$ which satisfy the same inequality by Lemma~\ref{lem: sombrero} part $(2)$ as $f_i(z)/M_i$ takes values in $[-1,1]$ on $K$ by definition of $M_i$.
\item $P_m(z)=1$ if and only if $z\in X$. By $(1)$ $P_m(z)$ is equal to one if and only if all the summands $H_m\left(\frac{f_i(z)}{M_i}\right)$ assume the value one which by Lemma~\ref{lem: sombrero} occurs if and only if $f_i(z)=0$ for $i=1,\dots, s$ or equivalently if and only if $z\in X$.

\item 
Now let $\eta, D$ be the positive real numbers given by Lemma~\ref{lem: genpoints} and let $m$ be even and sufficiently large. For $z\in K$ we have one of the following cases

\begin{enumerate}
\item There exists an index $j$ such that $\left|\arcsin(f_j(z)/M_j)\right|\geq \frac{2}{m}$. In this case, Lemma~\ref{lem: sombrero} part $(4a)$ implies that $\left|H_m\left(\frac{f_j(z)}{M_j}\right)\right|\leq 3/4$ and therefore we have
\[P_m(z)\leq   1-\frac{1}{s}+\frac{3}{4s} =1-\frac{1}{4s}\] 
\item For every index $j$ the inequality $\left|\arcsin(f_j(z)/M_j)\right|\leq \frac{2}{m}$ holds. Then by Lemma~\ref{lem: sombrero} part $(4b)$ we have
\[P_m(z)\leq \frac{1}{s}\sum_{t=1}^s \left(1-\frac{m^2}{3!2}\left(\frac{f_t(z)}{M_t}\right)^2\right)\]
By our assumption on $m$ this implies that $z\in N(X,\eta)$ and therefore by Lemma~\ref{lem: genpoints} we conclude that
\[P_m(z)\leq 1-\frac{m^2}{3!2s}D d(X,z)^2.\] 
\end{enumerate}
\end{enumerate}
As a result the inequality
\[P_m(z)\leq \max\left(1-\frac{m^2}{3!2s}Dd(X,z)^2, 1-\frac{1}{4s}\right)\]
holds for every $z\in K$, proving a QIC condition with constants $C_a:=\frac{m^2D}{3!2s}$ and $C_b=\frac{1}{4s}$. The conclusions of part $(1)$ and $(2)$ of the Theorem follow by applying Theorem~\ref{thm: Approx}.

$(3)$ Suppose $x^*\in X$ and define the polynomial
\[G_m(z)=\frac{1}{n}\left(\sum_{j=1}^n H_{m}\left(\frac{z_j-x^*_j}{{\rm diam}(K)}\right)\right)\]
where ${\rm diam}(K)$ is the largest distance between any two points in $K$. Note that $G_m(z)$ is a polynomial of degree $m\leq d$, that $\|G_m\|_{\infty}=1$ on $K$ and that  $G_m$ achieves the maximum value one only when $z=x^*$. Using $G_m$ we can re-write $\left| \mu(x^*)-\sum_{z: d(z,x^*)\leq c_0}\hat\Delta(z)\right|$ as 
\[\left| \left(\mu(x^*)-\int_K G_md\mu\right) + \left(\int_K G_md\mu-\int_KG_md\hat\Delta\right) +\left(\int_KG_md\hat\Delta-\sum_{z: d(z,x^*)\leq c_0}\hat\Delta(z)\right)\right|.\]
the claim will be proven by using the triangle inequality and bounding the absolute values of each of the terms as follows:
\begin{enumerate}
\item $\left| \left(\mu(x^*)-\int_K G_md\mu\right)\right|\leq \frac{1}{m}\|\mu\|_{\rm TV}$. This is because $m$ is sufficiently large so that $|G_m(z)|$ is bounded by $\frac{1}{m}$ at all points of $X$ distinct from $x^*$. 
\item $\left|\int_K G_m(d\mu-d\hat\Delta)\right|\leq 2\delta$. This is a consequence of Lemma~\ref{lem: key} part $(1)$ since $\|G_m\|_{\infty}\leq 1$ on $K$.
\item To bound the term $\left|\int_KG_md\hat\Delta-\sum_{z: d(z,x^*)\leq c_0}\hat\Delta(z)\right|$ we rewrite the left-hand side and use the triangle inequality obtaining an upper bound of 
\[\left|\sum_{z: d(z,x^*)\leq c_0, \Delta(z)>0}(1-G_m(z))\hat\Delta(z)\right| + \left|\sum_{z: d(z,x^*)\geq c_0, \Delta(z)>0 }G_m(z)\hat\Delta(z)\right|+\left|G(z)\sum_{\Delta(z)<0} \hat\Delta(z) \right| \]
The second and third term are bounded above by $2(s+1)\delta$ by what we have proven in part $(2)$ and the fact that $|G(z)|\leq 1$.
For the remaining term note that Lemma~\ref{lem: sombrero} part $(5c)$ implies that the following inequality holds for every $z\in K$,
\[ G_m(z)\geq \frac{1}{n}\sum_{j=1}^n\left(1-\frac{m^2 \frac{\pi^2}{2}(z_j-x^*_j)^2}{{\rm diam}(K)^2} \right)= 1-\frac{m^2 \frac{\pi^2}{2}d(z,x^*)^2}{{\rm diam}(K)^2}.\] 
Since $d(z,x^*)\leq c_0$ we can assume $m$ is sufficiently large so that $d(z,x^*)=d(z,X)$ and conclude that the term $\left|\sum_{z: d(z,x^*)\leq c_0, \Delta(z)>0}(1-G_m(z))\hat\Delta(z)\right|$ is bounded above by
\[\sum_{z: d(z,x^*)\leq c_0, \Delta(z)>0} \frac{m^2 \frac{\pi^2}{2}}{{\rm diam}(K)^2} d(z,X)^2\hat{\Delta}(z)\leq \left(\frac{3!4s \frac{\pi^2}{2}}{{\rm diam}(K)^2D}\right)\delta.\]
where the last inequality follows from what we have proven in part $(2)$.
\end{enumerate}
Combining the above inequalities we conclude that
\[\left| \mu(x^*)-\sum_{z: d(z,x^*)\leq c_0}\hat\Delta(z)\right|\leq \frac{1}{m}\|\mu\|_{\rm TV} + \left(2(s+1)+ \frac{3!4s \frac{\pi^2}{2}}{{\rm diam}(K)^2D}\right)\delta\]
as claimed.
\end{proof}

\begin{remark} The previous Theorem is better expressed in words: parts $(1)$ and $(2)$ prove that if there are large recovered spikes then these must lie near true spikes. Part $(3)$ shows the complementary statement that there {\it must} exist recovered spikes close to the true spikes.
\end{remark}

\begin{remark}
In the setting of one-dimensional super-resolution there are stronger quantitative guarantees than~\cite{AzaisDeCastroGamboa} under additional assumptions (see for instance~\cite{FG3,DP}). It is an interesting question to ask whether those guarantees can be extended to the higher-dimensional setting we consider in this paper.
\end{remark}

The following example shows that the constants in the previous proofs can be computed if one has a sufficiently precise understanding of the geometry of the support set $X$.
\begin{example}\label{Exo: Grids} (Measures supported on a finite grid). Suppose $A_1,\dots, A_n\subseteq [0,1]\subseteq  \RR$ are finite sets of size $s>1$. Let $X=A_1\times A_2\times \dots \times A_n\subseteq K:=[0,1]^n$. In this example we will estimate the constant $D$ appearing in Lemma~\ref{lem: genpoints}. Let $d_i$ be the minimum distance among points with distinct $i$-th coordinate projection, that is $d_i:=\min_{a\neq b \in A_i\cup \{0,1\}} |a-b|$ and let $d_{\min} = \min_{i=1,\dots, n} d_i$. The ideal $I(X)$ is defined by the polynomials $p_i(x):=\prod_{a\in S_j}(x_i-a)$ for $i=1,\dots,n$. Suppose $A_i\cup \{0,1\}$ is given by $0=\alpha_0^i<\alpha_1^i<\dots <\alpha_{s}^i<\alpha_{s+1}^i=1$. For $j=0,\dots, s$ define $\beta_{j}^i=(\alpha_j^i+\alpha_{j+1}^i)/2$ and note that, whenever $x_i\in [\beta_{j-1}^i,\beta_{j}^i]$ the following inequality holds

\[p_i^2=(x_i-\alpha_j^i)^2\left(\prod_{t\neq j}(x_i-\alpha_t^i)^2\right)\geq (x_i-\alpha_j^i)^2 \left(\frac{d_i}{2}\right)^{2(s-1)}\]

For $i=1,\dots, n$ let $M_i=\max_{x\in [0,1]^n} p_i(x)$, let $M=\max M_i$ and note that $M\leq 1$. For any set of choices $j(i)\in \{1,\dots, s\}$ and all points in the box $\prod_{i=1}^n [\beta_{j(i)-1}^i,\beta_{j(i)}^i]\subseteq K$ the following inequalities hold:

\[h(x)=\sum_{i=1}^n \left(\frac{p_i}{M_i}\right)^2 \geq \frac{1}{M^2}\sum_{j=1}^n (x_j-\alpha_{j(i)}^i)^2\left(\frac{d_i}{2}\right)^{2(s-1)}\geq \left(\frac{d_{\min}}{2}\right)^{2(s-1)} d(x,T)^2\]

Since the rightmost inequality is independent of $i$ we can therefore choose $D:=\left(\frac{d_{\min}}{2}\right)^{2(s-1)}$ in Lemma~\ref{lem: genpoints} and obtain a completely explicit quantitative recovery guarantee from Theorem~\ref{thm: ApproxMain} for measures supported on grids.
\end{example}
Specializing the previous example to the one dimensional case, Theorem~\ref{thm: ApproxMain} parts $(2)$ and $(3)$ implies that $O(m)$ measurements perturbed by noise of magnitude $\delta$ lead to a recovered measure $\hat{\Delta}$ such that, if $\beta=d_{\rm min}$ and $c_0=O\left(\frac{1}{m\beta^{s-1}}\right)$ then:
\begin{enumerate}
\item The recovered spikes far from the true spikes are small: $\sum_{z\in F(X,z_0), \hat{\Delta}(z)>0} \hat{\Delta}(z)= O(2s\delta)$ and
\item The recovered spikes near the true spikes are big, that is for every $x^*\in X$
\[\left|\mu(x^*)-\sum_{z: d(z,x^*)\leq c_0} \hat{\Delta}(z) \right|= \frac{1}{m}\|\mu\|_{\rm TV} + O\left(\frac{s\delta}{\beta^{(2s-2)}}\right)\].
\end{enumerate}

There is a significant amount of work in trying to understand the accuracy of one-dimensional super resolution, especially in the Fourier case (In approximate super-resolution the choice of basis for the space $V$ is very important because the information we are given on the magnitude of the error depends on this choice). It is a very interesting direction for further research to determine which, if any, of the following results can be extended to variations of the higher-dimensional setting considered in this article:

\begin{enumerate}
\item Assuming a probabilistic model for the error the right-hand side of part $(2)$ above can be improved to $O(\delta)$ as in~\cite[Theorem 2.2]{AzaisDeCastroGamboa}. Extending this result would require developing a higher-dimensional version of the Bernstein isolation property~\cite[Definition $2.3$]{AzaisDeCastroGamboa}.

\item If $\mu=\sum_{i=1}^N c_i\delta_{x_i}$ is a point measure supported on a subset of size $k$ (i.e. with only $k$ nonzero coefficients $c_i$) of an equally spaced {\it fixed grid} of size $N$ in $[-1,1]$ then the problem of recovering $\mu$ becomes an instance of compressive sensing with respect to the measurement matrix $A\in \RR^{N\times T}$ which maps the vector of coefficients $c\in \RR^N$ to the moments of $\mu$ with respect to the basis functions $\phi_i$.

In this setting, the key quantity of interest is the minimax error, defined as the worst-case error of the best recovery algorithm, namely:
\[E(k,\delta):=\inf_{\hat{c}} \sup_{\text{k-sparse c}}\sup_{e: \|e\|_2\leq \delta} \|\hat{c}-c\|_2\]
where $\hat{c}$ is any algorithm to recover the true coefficients $c$ of the measure from the noise-corrupted vector of moments $f=Ac+e$. It is known~\cite[Theorem 1]{DemanetNguyen} that this quantity is controlled by the number $\epsilon_{2k}:=\min_{T:|T|=2k}\sigma_{\min}(A_T)$, that is by the smallest singular value of a submatrix consisting of $2k$ columns of $A$, in the sense that
\[\frac{1}{2\epsilon_{2k}}\delta\leq E(k,\delta)\leq \frac{2}{\epsilon_{2k}}\delta\]
There are several results about the limits of superresolution in this setting. It is known~\cite{DonohoGrid, DemanetNguyen, LiLiao} that the best possible error rate for one-dimensional super resolution in the Fourier basis is $O\left(\frac{1}{\sqrt{m}(m\beta)^{2(k-1)}}\delta\right)$ where $\beta$ is the minimum distance between distinct grid points. More precise estimates are available when further geometric assumptions are made on the distribution of the support points on the grid~\cite{LiLiao, Batenkov, Batenkov2, LiLiao2}.
Extending these sharp results to the higher-dimensional polynomial setting would first require a natural choice of basis (since the quantities $\epsilon_{2k}$ are obviously basis dependent). We believe it is interesting to study the behavior of higher-dimensional polynomial superresolution on a basis given by a random sample of Kostlan-Shub-Smale polynomials as in~\cite{GHJV0}.

\end{enumerate}

\begin{remark}
Note that none of the cited results are directly comparable to ours since they use very different assumptions, either a fixed probabilistic model for the noise or a fixed basis or the assumption that our unknown measures are supported on $k$-sparse subsets of a fixed finite grid. We prefer not make these assumptions since they are not adequate for our current applications as described in the Introduction and in Section~\ref{Sec: numerics}.
\end{remark}

\subsection{An algorithm for approximate super-resolution}

In this section we focus on solving problem~(\ref{prob: BLASSO}). We begin by proving Theorem~\ref{thm: dual} which reformulates~(\ref{prob: BLASSO}) as a finite-dimensional convex optimization problem amenable to computation whenever~(\ref{prob: BLASSO}) has a discrete minimizer.

\begin{proof}[Proof of Theorem~\ref{thm: dual}] During the proof we will identify problem~(\ref{prob: BLASSO2}) with the dual of~(\ref{prob: BLASSO}) and prove that there is no duality gap. To do this we first reformulate~(\ref{prob: BLASSO}) as a primal problem in standard form (as in~\cite[Section 7.1]{BarvinokConvex}). Recall that a signed Radon measure $\nu$ admits a unique Hahn decomposition as a difference of Radon measures $\nu_+$ and $\nu_-$ and that in this decomposition the total variation is given by $\|\nu\|_{\rm TV}=\nu_-(K)+\nu_{+}(K)$ which is a linear function in $\nu_+$,$\nu_-$. The ambient vector space of our primal optimization problem will be $E=C(K)^*\times C(K)^*\times \RR^{m}\times \RR$ endowed with the weak $\ast$-topology. We will denote its elements by $4$-tuples $(\nu_{-},\nu_{+},\vec{z},w)$. Define the convex cone
\[D:=\left\{(\nu_{-},\nu_{+},\vec{z},w):\|\vec{z}\|_2\leq w\text{ and $\nu_+,\nu_-\in R(K)_+$}\right\}\]
where $R(K)_+$ denotes the cone of positive radon measures on $K$. The continuous dual of $E$, denoted $E^*$ is given by $E^*:=C(K)\times C(K)\times \RR^{m}\times \RR$ and we will write its elements as $4$-tuples $(f_1,f_2,\vec{a},b)$. In this notation the dual cone $D^*\subseteq E^*$ is given by:
\[D^*:=\left\{(f_1,f_2,\vec{a},b): \|\vec{a}\|_2\leq b\text{ and $f_1,f_2\geq 0$ on $K$} \right\}.\] 
To simplify the notation we will write $\int_K fd\nu:=\langle f,\nu\rangle$.
Define the continuous linear map $A: E\rightarrow \mathbb{R}^m\times \RR$ by the formula 
\[A(\nu_{-},\nu_{+},\vec{z},w) = \left(\left(\langle\phi_i,\nu_+-\nu_-\rangle-z_i\right)_{i=1,\dots, m},w\right)\]
and note that problem~(\ref{prob: BLASSO}) is equivalent to
\[ \min_{(\nu_{-},\nu_{+},\vec{z},w)\in D} \langle 1,\nu_++\nu_{-}\rangle\text{ s.t. $A(\nu_{-},\nu_{+},\vec{z},w)= \left(\vec{y},\delta\right)$}\]
its dual problem is therefore given by (see~\cite[Section 7.1]{BarvinokConvex}
\[ \sup_{(f_1,f_2,\vec{a},b)} \langle \vec{a},y\rangle +\delta b\text{ s.t. $(1,1,0,0)-A^*(\vec{a},b)\in D^*$} .\]
By definition of adjoint we have $A^*(\vec{a},b)=\left(\langle \vec{a},\vec{\phi}\rangle,-\langle \vec{a},\vec{\phi}\rangle, -\vec{a},b\right)$ so the dual is equivalent to~(\ref{prob: BLASSO2}) after the change of variables $b\rightarrow -b$. To prove the Theorem we will show that there is no duality gap. Since the objective function is nonnegative and the domain of the problem is nonempty (because its feasible set contains the measure $\mu$ which we would like to recover) by~\cite[Theorem 7.1]{BarvinokConvex} it suffices to prove that $\hat{A}(D)\subseteq \RR^{m+2}$ is closed where $\hat{A}: E\rightarrow \RR\times \RR^{m}\times \RR$ is given by 
\[\hat{A}(\nu_{-},\nu_{+},\vec{z},w)=\left( \langle 1,\nu_++\nu_-\rangle, A(\nu_{-},\nu_{+},\vec{z},w)\right).\]
Assume $\{\beta_j\}_j$ with $\beta_j=(\nu_{-}^j,\nu_{+}^j,\vec{z}^j,w^j)$ is a sequence in $D$ for which $\hat{A}(\beta_j)$ converges to $s\in \RR^{m+2}$ as $j\rightarrow \infty$. We will show that there exists $\beta\in D$ such that $\hat{A}(\beta)=s$. Since $(\hat{A}(\beta_j))_j$ is a convergent sequence in $\RR^{m+2}$ it is bounded and therefore both the total variation of the $\nu_{\pm}^j$ and the $w^j$ which are the first and last components of the map $\hat{A}$ are bounded. By the Theorem of Banach-Alaoglu we know that balls in $C(K)^*$ are compact in the weak $\ast$ topology and therefore conclude that the points $\beta_j$ lie in a compact subset of the closed cone $D\subseteq E$. As a result there is a subsequence $(\beta_{j_n})_n$ converging to a point $\beta\in D$. Since $\hat{A}$ is a continuous linear map we conclude that $\hat{A}(\beta)=s$ as claimed.

\end{proof}

\begin{remark}\label{rem: explicit} If we think of a signed measure as a linear operator $L\in V^*$ in the ellipsoid $\mathcal{E}$ defined by $\left\|(L(\phi_i)-y_i)_{i=1\dots m}\right\|_2\leq \delta$ then the quantity $-\|\vec{a}\|_2\delta+ \sum_{i=1}^m a_iy_i$ equals $\inf_{L\in \mathcal{E}} L(P)$ where $P:=\sum a_i\phi_i$ and the optimization problem above can be thought of as solving
\[ \sup_{P: \|P\|_{\infty}\leq 1, P\in V}\left( \inf_{L\in \mathcal{E}} L(P)\right)\] 
This suggests a methodology for recovering an optimizer measure, given an optimal solution $(\vec{a}^*,\alpha^*)$ of~(\ref{prob: BLASSO2}), namely:
\begin{enumerate}
\item Define $P^*:=\sum a_i^*\phi_i$ and find an operator $L^*$ which is a minimizer of the second-order cone optimization problem $\inf_{L\in \mathcal{E}} L(P^*)$.
\item The values $L^*(\phi_i)$ are the moments of a measure which we can try to recover via exact superresolution as in the previous section. The moments of this measure are contained in $\mathcal{E}$ and $L^*(P^*)=\alpha^*$ so the measure has total variation $\alpha^*$ and is therefore a minimizer of~(\ref{prob: BLASSO}). 
\end{enumerate}
\end{remark}

Next we prove Theorem~\ref{thm: hierarchy} which gives a semidefinite programming hierarchy for solving~(\ref{prob: BLASSO2}) on explicitly bounded semialgebraic sets.

\begin{proof}[Proof of Theorem~\ref{thm: hierarchy}] $(1)$ A polynomial $h$ is a sum-of-squares of polynomials of degree at most $e$ iff there exists a PSD matrix $A$ such that $h=\vec{m}^tA\vec{m}$ where $\vec{m}$ is the vector of monomials of degree at most $e$. It follows that $Q_s(g)$ is a Semidefinitely Representable (SDR) set (i.e. a linear projection of a spectrahedron) for any $s>0$. We conclude that the set 
\[\left\{(\vec{a},b): 1-\langle \vec{a},\vec{\phi}\rangle, 1+\langle \vec{a},\vec{\phi}\rangle \in Q_s(g), \|\vec{a}\|_2\leq b\right\}\]
is also SDR since it is an intersection of two affine slices of SDR sets and a second-order cone constraint. Since the function $\langle\vec{a},y\rangle -b\delta$ is linear on $(\vec{a}, b)$ we conclude that $\alpha_s$ is the optimal value of a semidefinite programming problem as claimed.

$(2)$ Suppose that $(\vec{a}^*,b^*)$ is an optimal solution of~(\ref{prob: BLASSO2}). For $\epsilon>0$  let $\vec{a}':=(1-\epsilon)\vec{a}^*$ and $b':=(1-\epsilon)b$. It is immediate that $1-\langle \vec{a}',\vec{\phi}\rangle >0$ and $1+\langle \vec{a}',\vec{\phi}\rangle >0$. Since $K$ is explicitly bounded Putinar's Theorem~\cite{Putinar} implies that there exists an integer $e>0$ such that $1-\langle \vec{a}',\vec{\phi}\rangle,1+\langle \vec{a}',\vec{\phi}\rangle\in Q_{e}(g)$ and therefore $\alpha_s$ is at least the optimal value at $(\vec{a}',b)$, that is $(1-\epsilon)\alpha$. We conclude that $(1-\epsilon)\alpha \leq\alpha_e\leq \alpha$ proving the claim since $\epsilon>0$ was arbitrary.

\end{proof}

We are now in a position to prove the summarization Theorem~\ref{thm: summary}.

\begin{proof}[Proof of Theorem~\ref{thm: summary}] Suppose $\Delta$ is a $(\delta,k)$ summary of $\mu$ and let $X:={\rm supp}(\Delta)$ Since the moments depend continuously on the location of the points we can assume, by slightly perturbing the support of $\Delta$, if necessary, that $X$ is a generic set of points.  If we define $y_i:=\int_K\phi_i d\mu$ then $y_i=\int_K\phi_id\Delta+\epsilon_i$ with $\|(\epsilon_i)_i\|_2\leq \delta$. Since $d\geq \max(2g(X),i(X))$ the claim follows from Theorem~\ref{thm: Approx}.
\end{proof}

\section{Numerical Experiments}
\label{Sec: numerics}

\subsection{Exact Recovery}
In this section we use the SDP procedure outlined in~Section~\ref{Sec: moments} to recover discrete measures in $K:=[-1,1]^n$, for $n = 1,2,4$ with $V=V_{\leq d}$. The goal is to record the behavior of the algorithm as $d$ and $k$ vary for measures supported on generic points. For each pair $(k,d)$ we generate $100$ uniform discrete measures $\Delta_j = \sum_{i=1}^k \frac{1}{k}\delta_{x_i^j}$ with support $S_j :=\{x_{1}^j,...,x_{k}^j\}$ in $[-1,1]^n$ chosen uniformly at random. For each $j$ we compute the moments with respect to the standard monomial basis of $V_{\leq d}$. To quantify the quality of the recovery we evaluate the function $q := H^*$ at the points $x_{i}^j$ and report the proportion of points where this quantity is very close to zero. Figure~\ref{exact_heat} reports the average of these proportions over the $100$ simulations. Figure~\ref{dim1} shows the function $H^*$ for degrees $d=2,3,4$ where $\Delta$ is a counting measure supported at four points in $[-1,1]$. Figure~\ref{dim2} shows the heatmap of the function $\log(H^*)$, for  degrees $d = 1,2,4,6$ where $\Delta$ is a counting measure supported in four points on $K=[0,1]\times [0,1]$. As expected, location accuracy increases with degree.

\begin{center}
\begin{figure}[H]
    \centering
    \begin{subfigure}[b]{0.4\textwidth}
        \includegraphics[width=\textwidth]{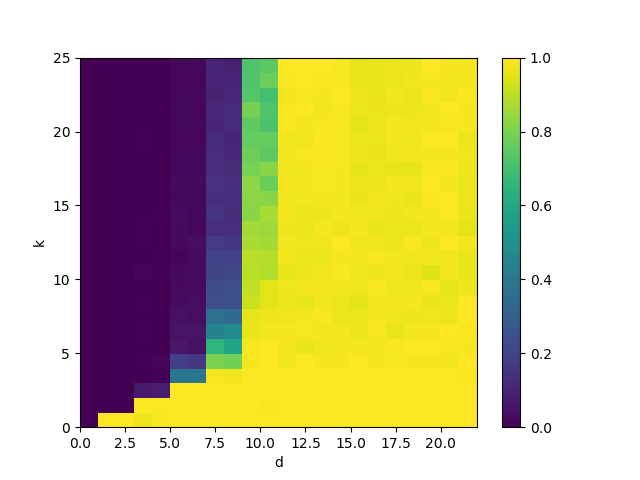}
        \caption{Recovery in dimension 1}
    \end{subfigure}
    ~ 
    \begin{subfigure}[b]{0.4\textwidth}
        \includegraphics[width=\textwidth]{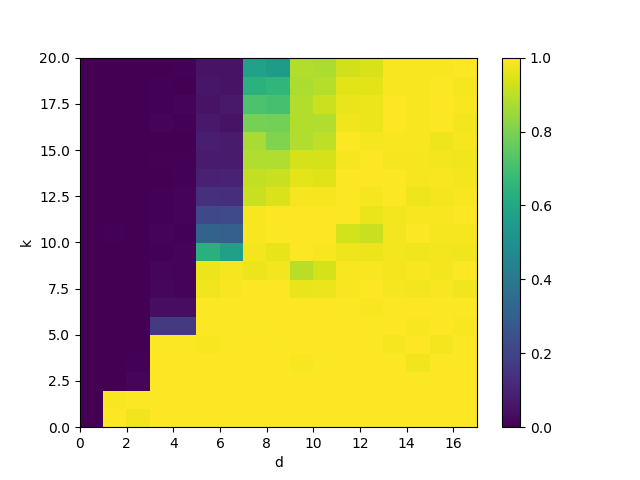}
        \caption{Recovery in dimension 2}
    \end{subfigure}\\
    ~ 
    \begin{subfigure}[b]{0.4\textwidth}
        \includegraphics[width=\textwidth]{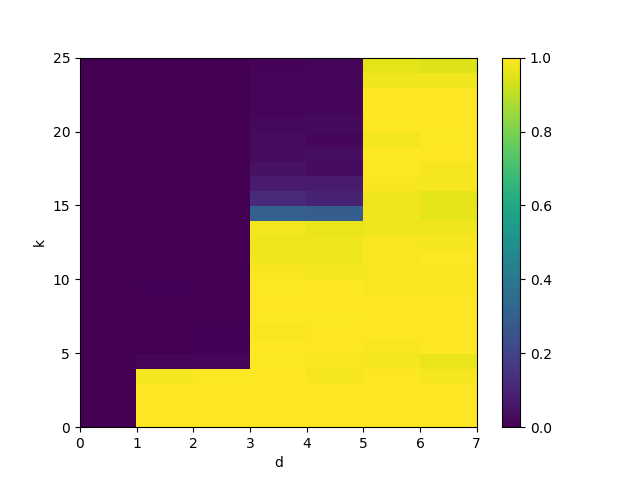}
        \caption{Recovery in dimension 4}
    \end{subfigure}
    \caption{Average revery for different dimensions and $100$ simulations}\label{exact_heat}
\end{figure}
\end{center}

\begin{center}
\begin{figure}[H]
    \centering
    \begin{subfigure}[b]{0.4\textwidth}
        \includegraphics[width=1\textwidth]{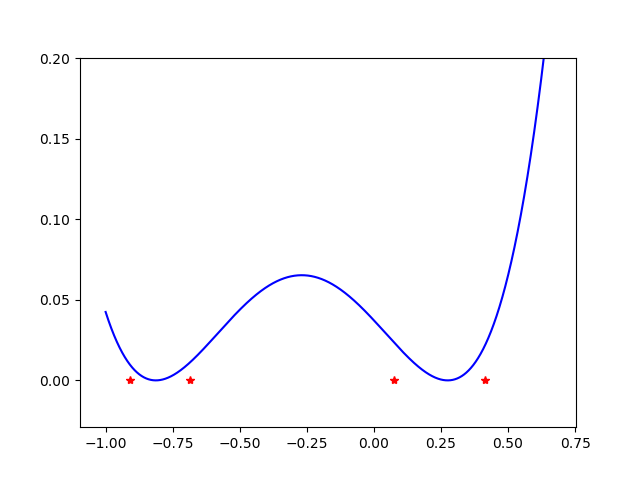}
        \caption{d=2}
        \label{d=2_dim1}
    \end{subfigure}
    \begin{subfigure}[b]{0.4\textwidth}
        \includegraphics[width=\textwidth]{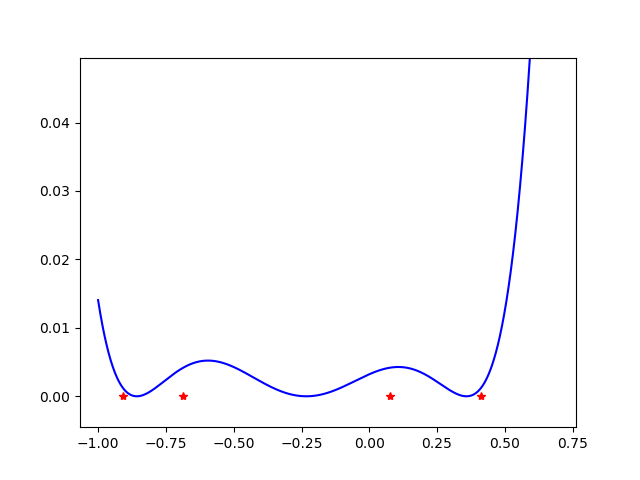}
        \caption{d=3}
        \label{d=3_dim1}
    \end{subfigure}
    ~ 
    \begin{subfigure}[b]{0.4\textwidth}
        \includegraphics[width=\textwidth]{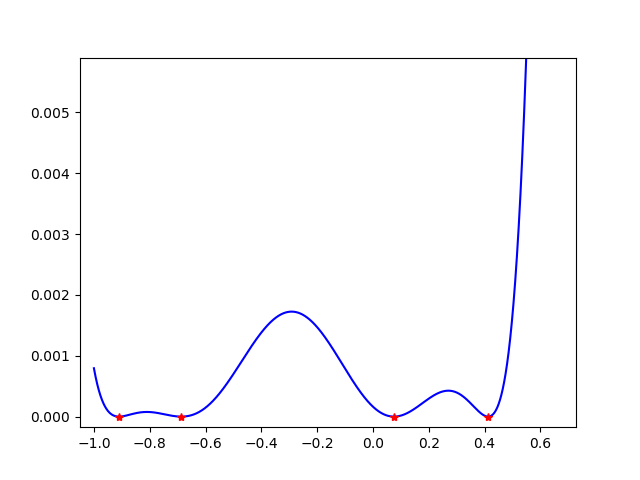}
        \caption{d=4}
        \label{d=4_dim1}
    \end{subfigure}
    \caption{Polynomial $H^*$ associated to the counting measure on $4$ points for different values of degree $d$ via the recovery procedure.}\label{dim1}
\end{figure}
\end{center}

\begin{center}
\begin{figure}[H]
    \centering
    \begin{subfigure}[b]{0.45\textwidth}
        \includegraphics[width=\textwidth]{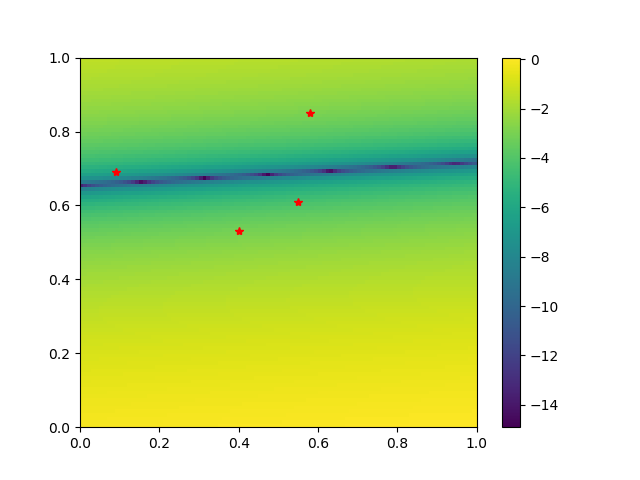}
        \caption{d=1}
        \label{d=1_dim2}
    \end{subfigure}
    ~ 
	    \begin{subfigure}[b]{0.45\textwidth}
        \includegraphics[width=\textwidth]{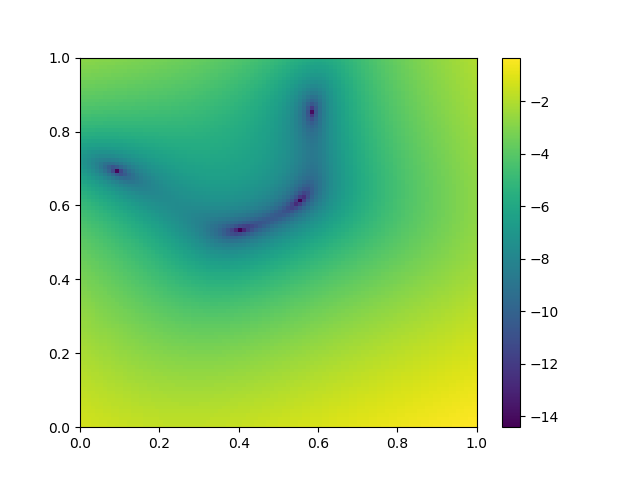}
        \caption{d=2}
        \label{d=2_dim2}
    \end{subfigure}      
\end{figure}    
\begin{figure}[H]\ContinuedFloat      
    \begin{subfigure}[b]{0.45\textwidth}
        \includegraphics[width=\textwidth]{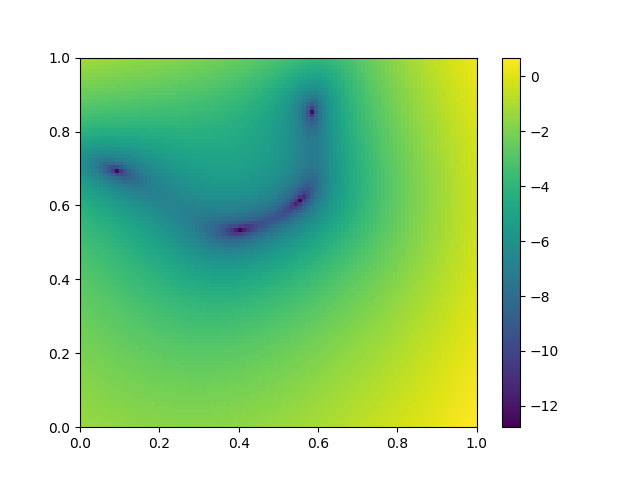}
        \caption{d=4}
        \label{d=4_dim2}
    \end{subfigure}
    ~ 
    \begin{subfigure}[b]{0.45\textwidth}
        \includegraphics[width=\textwidth]{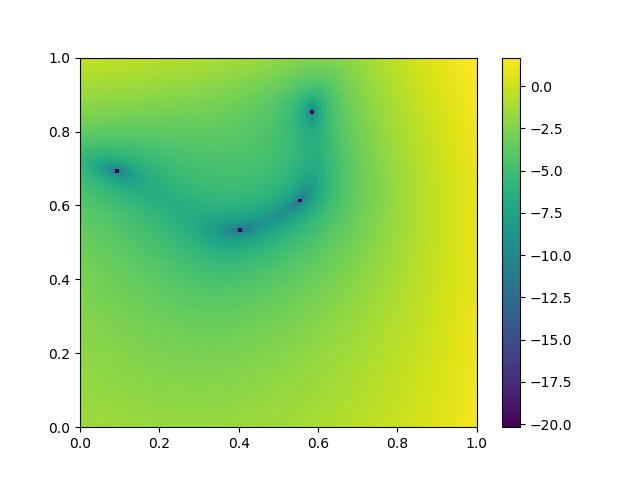}
        \caption{d=6}
        \label{d=6_dim2}
    \end{subfigure}
    \caption{Logarithm of $H^*$ when $\mu$ is a counting measure supported in four points and different values of the degree $d$.}\label{dim2}
\end{figure}
\end{center}

\subsection{Approximate Recovery}

We let $\mu$ be the counting measure supported on the five red points of Figures~\ref{dim1_approx} and ~\ref{dim2_approx} (in dimensions one and two respectively). Noisy measurements $y'_j= \int\Phi _j d\mu +\epsilon_i$ are generated, where $\epsilon_i$ is a sample with distribution $\mathcal{N}(0,\epsilon)$ and $\{\Phi_1,...,\Phi_m\}$ is the ortonormalization of the monomial basis of $V_{\leq d}$ with respect to the inner product given by the Lebesgue measure in $[-1,1]$ and $[0,1]^2$ for $d = 11$ and $d = 6$, respectively in dimension $1$ and $2$. We choose $\delta = \|(\epsilon_i)_i\|_2$ and use the hierarchy defined in~\ref{thm: hierarchy} with $e=d$.

\begin{center}
\begin{figure}[H]
    \centering
    \begin{subfigure}[b]{0.45\textwidth}
        \includegraphics[width=\textwidth]{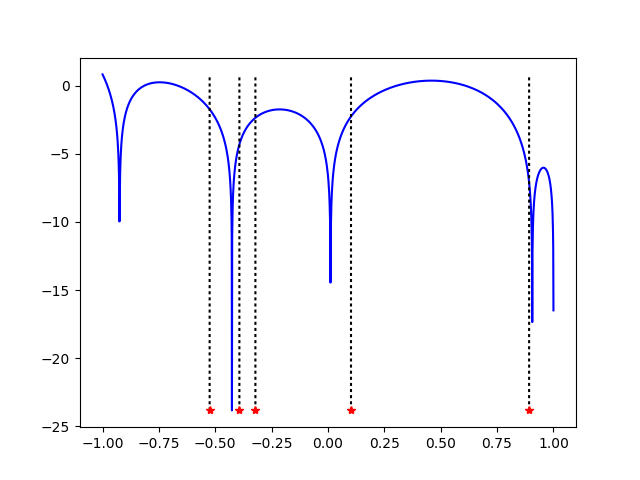}
        \caption{$\epsilon = 10^{-1}$}
        \label{eps_e-1}
    \end{subfigure}
 ~   
    \begin{subfigure}[b]{0.45\textwidth}
        \includegraphics[width=\textwidth]{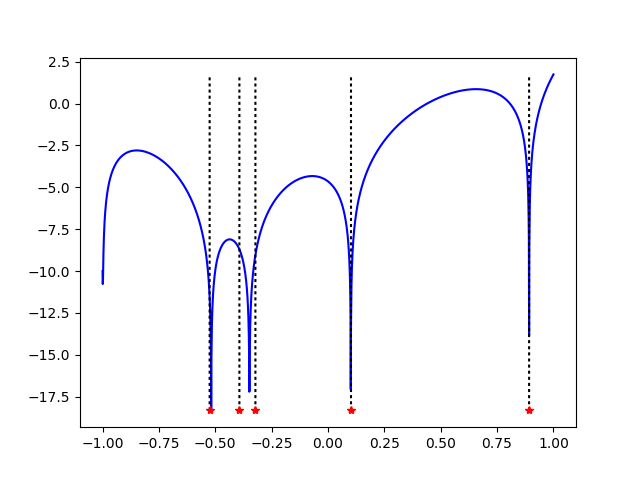}
        \caption{$\epsilon = 10^{-3}$}
        \label{eps_e-3}
    \end{subfigure}
\end{figure}
\begin{figure}[H]\ContinuedFloat
\centering
    \begin{subfigure}[b]{0.45\textwidth}
        \includegraphics[width=\textwidth]{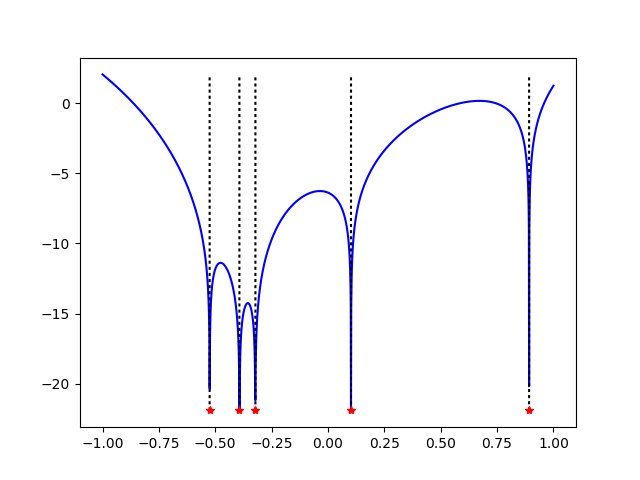}
        \caption{$\epsilon = 10^{-5}$}
        \label{eps_e-5}
    \end{subfigure}
    \caption{Logarithms of optimal polynomials $H^*$ for $d=11$, noisy measurements and varying $\epsilon$.}\label{dim1_approx}
\end{figure}
\end{center}

\begin{center}
\begin{figure}[H]
    \centering
    \begin{subfigure}[b]{0.45\textwidth}
        \includegraphics[width=\textwidth]{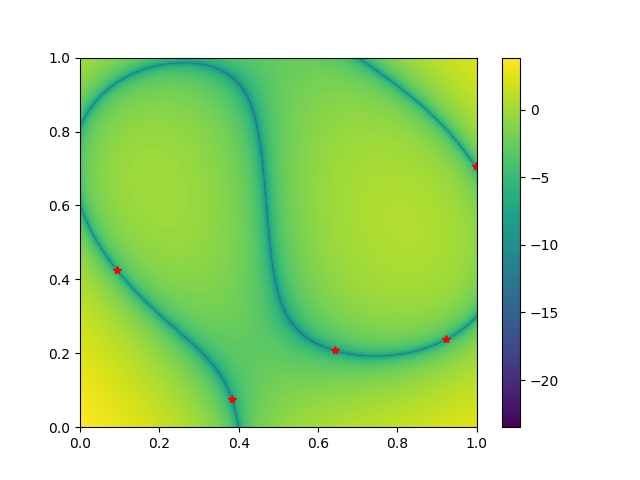}
        \caption{$\epsilon = 10^{-3}$}
        \label{eps_e-3_dim2}
    \end{subfigure}
   ~
    \begin{subfigure}[b]{0.45\textwidth}
        \includegraphics[width=\textwidth]{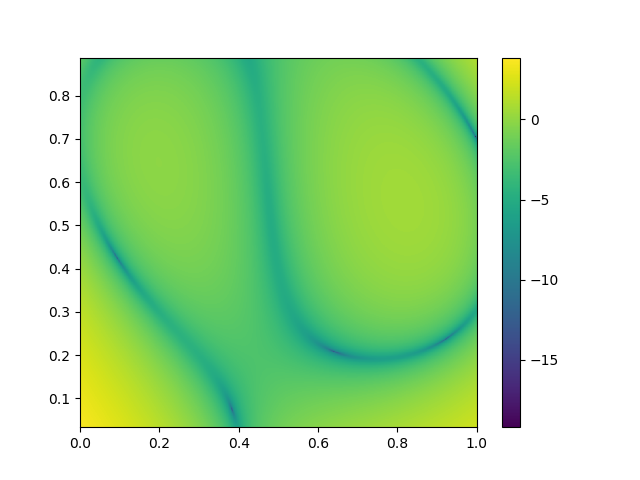}
        \caption{$\epsilon = 10^{-5}$}
        \label{eps_e-5_dim2}
    \end{subfigure}	
    ~
    \begin{subfigure}[b]{0.45\textwidth}
        \includegraphics[width=\textwidth]{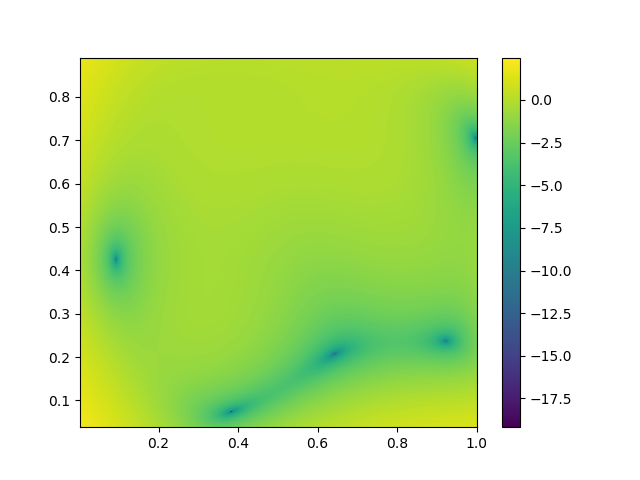}
        \caption{$\epsilon = 10^{-7}$}
        \label{eps_e-7_dim2}
    \end{subfigure}
    \caption{Logarithms of optimal polynomials $H^*$ for $d=6$, noisy measurements and varying $\epsilon$.}\label{dim2_approx}
\end{figure}
\end{center}

\subsection{Measure summarization}

Applying Theorem~\ref{thm: summary} to the Lebesgue measure on the interval $[-1,1]$ with $V=V_{\leq d}$, we obtain a very good approximation of the $d-th$ Gauss-Legendre nodes as local minima of the optimal polynomial $H^*$. This is illustrated in Figure~\ref{dim1_summary_leg}. The vertical lines correspond to the location of the Gauss-Legendre nodes.

\begin{figure}[H]
    \centering
    \begin{subfigure}[b]{0.45\textwidth}
        \includegraphics[width=\textwidth]{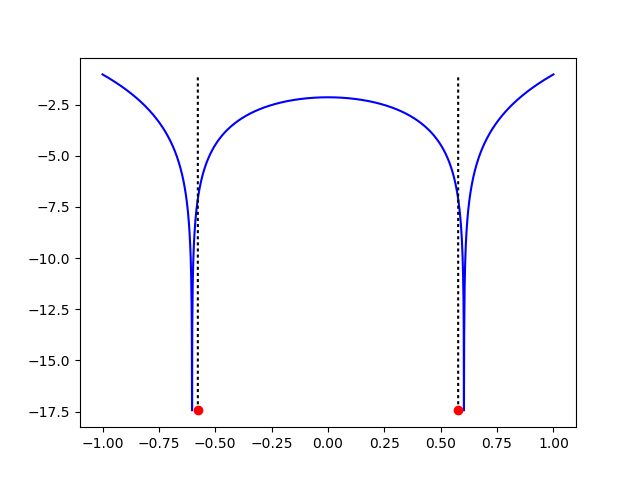}
        \caption{$d = 2$}
        
    \end{subfigure}
 	~    
    \begin{subfigure}[b]{0.45\textwidth}
        \includegraphics[width=\textwidth]{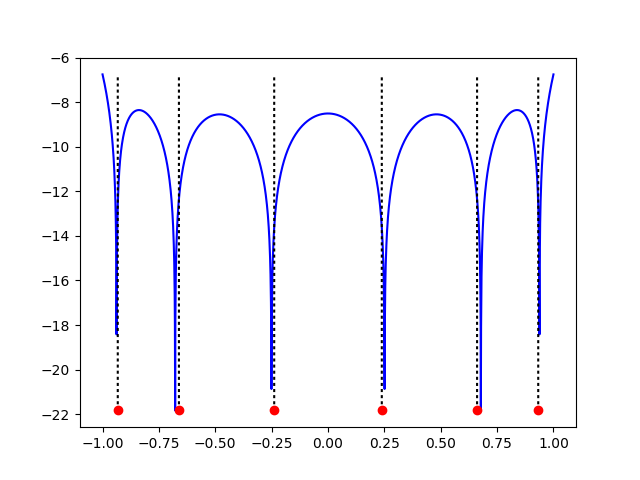}
        \caption{$d = 6$}
        
    \end{subfigure}\\
	~
    \begin{subfigure}[b]{0.45\textwidth}
        \includegraphics[width=\textwidth]{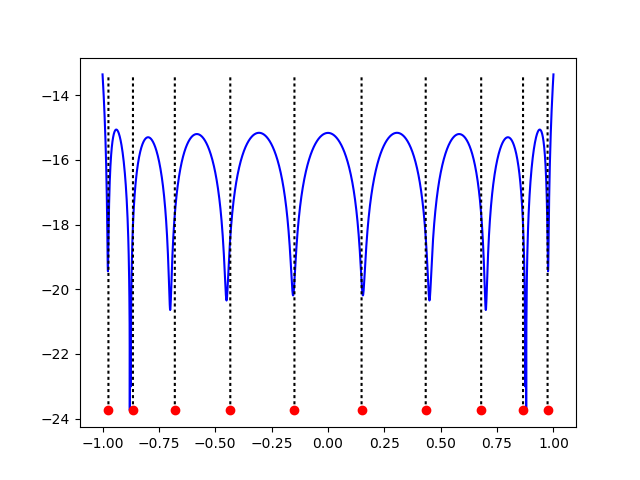}
        \caption{$d = 10$}
        
    \end{subfigure}
    \caption{Summarization of the uniform measure in $[-1,1]$}\label{dim1_summary_leg}
\end{figure}

Similarly we use Theorem~\ref{thm: summary} to obtain discrete approximations to the measures in $[0,1]$ given by the densities $w_1(x):=\sqrt{1-x^2}$ and $w_2(x):= \frac{1}{\sqrt{1-x^2}}$. The results are shown in Figure~\ref{dim1_summary}. The recovered measures turn out to be supported on a set very close to the roots of Chebyshev polynomials (marked in red) of degree $d$, which are known to lead to the best interpolation formulas~\cite[Section 6.1]{kincaid}. 

\begin{figure}[H]
    \centering
    \begin{subfigure}[b]{0.4\textwidth}
        \includegraphics[width=\textwidth]{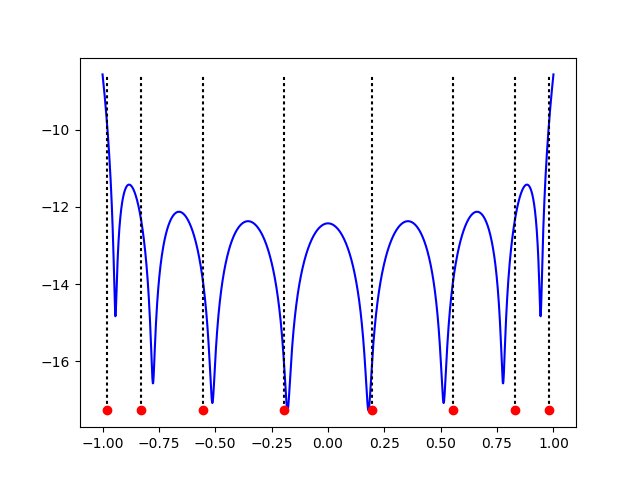}
        \caption{Density $w_1(x)$}
   
    \end{subfigure}
 	~    
    \begin{subfigure}[b]{0.4\textwidth}
        \includegraphics[width=\textwidth]{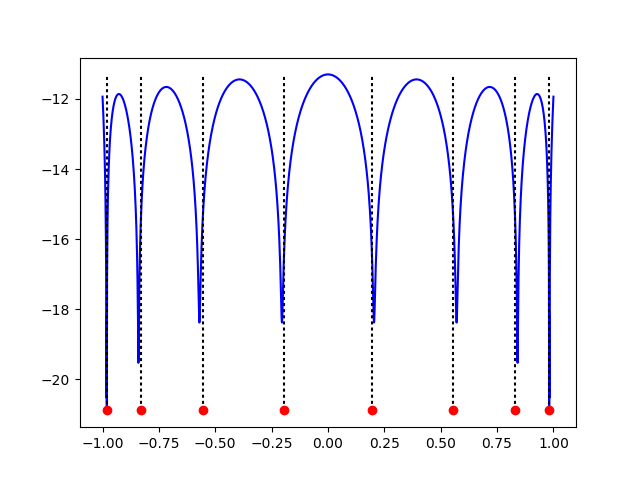}
        \caption{Density $w_2(x)$}
       
    \end{subfigure}
	
    \caption{}\label{dim1_summary}
\end{figure}

Finally, in Figure~\ref{dim2_summaries} we apply Theorem~\ref{thm: summary} to the Lebesgue measure over the square $[-1,1]^2$ with $V=V_{\leq d}$ for $d =3,4,5$. Note that when  $d=3,5$ the obtained summary is not the product measure of the one-dimensional summaries since its support contains $(0,0)$ (compare with Figure~\ref{dim1_summary_leg}). When $d=4$ the algorithm finds an $H^*$ with infinitely many real zeroes and is therefore unable to locate the support of a discrete summary. It would be interesting to find criteria which guarantee that problem~(\ref{prob: Summary}) has discrete minimizers (see ~\cite{BL} for some results on this problem for Fourier moments of complex radon measures in the torus).

\begin{figure}[H]
    \centering
    \begin{subfigure}[b]{0.4\textwidth}
        \includegraphics[width=\textwidth]{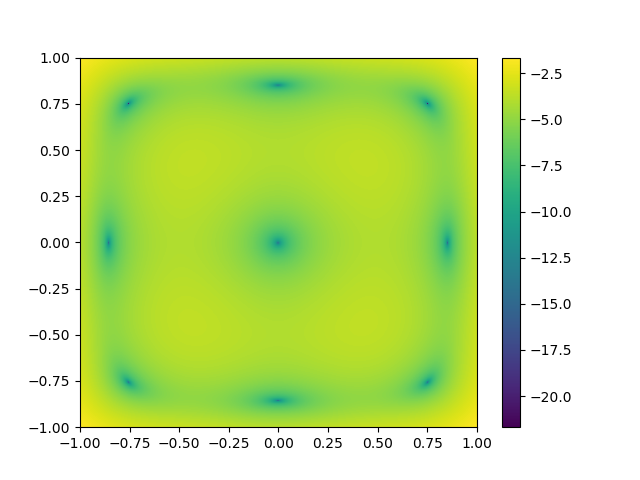}
        \caption{Nodes for $d = 3$.}
   
    \end{subfigure}
 	~    
    \begin{subfigure}[b]{0.4\textwidth}
        \includegraphics[width=\textwidth]{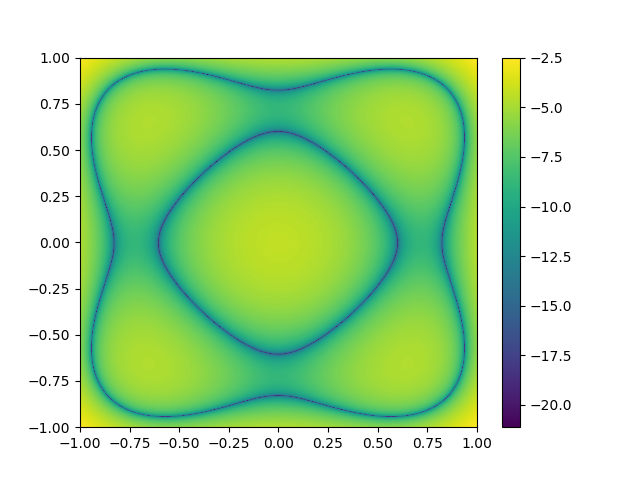}
        \caption{Nodes for $d = 4$.}
       
    \end{subfigure}\\
 	~    
    \begin{subfigure}[b]{0.4\textwidth}
        \includegraphics[width=\textwidth]{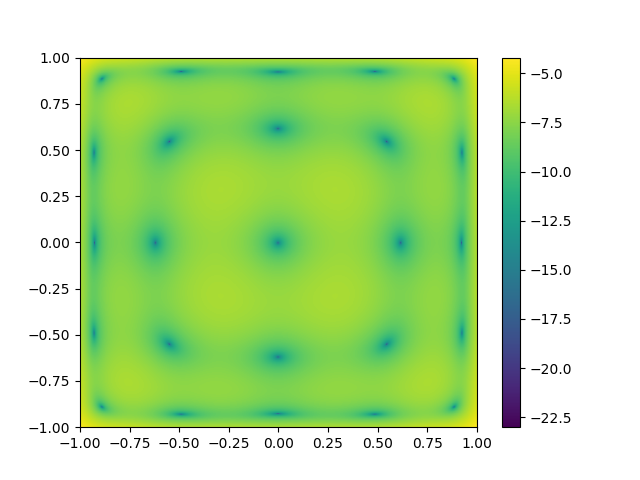}
        \caption{Nodes for $d = 5$.}
       
    \end{subfigure}
	
    \caption{}\label{dim2_summaries}
\end{figure}

All computations in this section were made with the Julia programming language~\cite{Julia} using the specialized solver~\cite{Mosek} and the JuMP modeling language~\cite{Jump}. The code used to generate the examples in this section is freely available at \url{https://github.com/hernan1992garcia/super_resolution_recovery}.

\bibliography{BibliographyNotesCompMR}

\end{document}